\documentclass[11pt]{amsart}
\usepackage{amssymb,latexsym,amsmath,graphicx,graphics,epic,eepic}

\theoremstyle{plain}
\newtheorem{theorem}{Theorem}
\numberwithin{theorem}{section}
\newtheorem{lemma}[theorem]{Lemma}

\newtheorem{corollary}[theorem]{Corollary}

\theoremstyle{definition}
\newtheorem{definition}[theorem]{Definition}
\newtheorem{example}[theorem]{Example}

\newtheorem{problem}[theorem]{Problem}

\newcommand{\C}{{\mathbb C}}
\newcommand{\R}{{\mathbb R}}
\newcommand{\Z}{{\mathbb Z}}
\newcommand{\Q}{{\mathbb Q}}
\renewcommand{\P}{{\mathbb P}}
\newcommand{\s}{{\mathbb S}}

              \newcommand{\E}{{\mathcal E}}

\begin{document}
\title{On a class of symplectic $4$-orbifolds with vanishing canonical class}
\author{Weimin Chen}
\thanks{}
\date{\today}
\maketitle

\begin{abstract}
A study of certain symplectic $4$-orbifolds with vanishing canonical class is initiated. We show that for any such symplectic $4$-orbifold $X$, there is a canonically constructed symplectic $4$-orbifold $Y$, together with a cyclic orbifold covering $Y\rightarrow X$, such that $Y$ has at most isolated Du Val singularities and a trivial orbifold canonical line bundle. The minimal resolution of $Y$, to be denoted by $\tilde{Y}$, is a symplectic Calabi-Yau $4$-manifold endowed with a natural symplectic finite cyclic action, extending the deck transformations of the orbifold covering $Y\rightarrow X$. Furthermore, we show that when $b_1(X)>0$, $\tilde{Y}$ is a $T^2$-bundle over $T^2$ with symplectic fibers,
and when $b_1(X)=0$, $\tilde{Y}$ is either an integral homology $K3$ surface or a rational homology $T^4$; in the latter case, the singular set of $X$ is completely classified. To further investigate the
topology of $X$, we introduce a general successive symplectic blowing-down procedure, which may be of independent interest. Under suitable assumptions, the procedure allows us to successively blow down a given symplectic rational $4$-manifold to $\C\P^2$, during which process we can canonically 
transform a given configuration of symplectic surfaces to a ``symplectic arrangement" of pseudoholomorphic curves in $\C\P^2$. The procedure is reversible; by a sequence of successive blowing-ups in the reversing order, one can recover the original configuration of symplectic surfaces up to a smooth isotopy. 
\end{abstract}

\section{Introduction and the main results}
In this paper, we consider a class of symplectic $4$-orbifolds which have vanishing canonical class. Our consideration has its origin in the study of symplectic Calabi-Yau $4$-manifolds endowed with certain symplectic finite group actions (cf. \cite{C}); in particular, the quotient orbifolds arising in \cite{C} belong to this class of $4$-orbifolds. (By definition, a symplectic $4$-manifold is called 
{\it Calabi-Yau} if it has trivial canonical line bundle.) We regard these symplectic $4$-orbifolds as certain intermediate objects, between the symplectic rational or ruled $4$-manifolds and the symplectic Calabi-Yau $4$-manifolds. On the one hand, we believe classifying such $4$-orbifolds is a more attainable objective, and on the other hand, we hope that these $4$-orbifolds may lead to
new progress in the topology of symplectic Calabi-Yau $4$-manifolds. Finally, these symplectic $4$-orbifolds are an interesting object to study in its own right. 

The class of symplectic $4$-orbifolds to be considered in this paper, which will be denoted by $X$ throughout, are specified by the conditions (i)-(iii) below. We denote the underlying space of 
$X$ by $|X|$. 
\begin{itemize}
\item [{(i)}] The canonical line bundle $K_X$, as an orbifold complex line bundle, has a well-defined first Chern class $c_1(K_X)\in H^2(|X|;\Q)$. We assume $c_1(K_X)=0$.
\item [{(ii)}] We assume the singular set of $X$ consists of a disjoint union of embedded surfaces 
$\{\Sigma_i\}$ and a set of isolated points $\{q_j\}$, where we denote by $m_i>1$ the order 
of isotropy along $\Sigma_i$, and by $G_j$ the isotropy group at $q_j$. Note that the symplectic $G_j$-action on the uniformizing system (i.e., orbifold chart) centered at $q_j$ naturally defines 
$G_j$ as a subgroup of $U(2)$ (i.e., $G_j\subset U(2)$). With this understood, we let $H_j$ be the normal subgroup of $G_j$ which consists of elements of determinant $1$ (i.e, $H_j=G_j\cap SU(2)$), and let $m_j$ be the order of the quotient group $G_j/H_j$, which is
easily seen cyclic. (The singular point $q_j$ is called a Du Val singularity if and only if $m_j=1$.) 
\item [{(iii)}] We set $n:=\text{lcm}\{m_i,m_j\}$ to be the least common multiple of $m_i,m_j$, and 
we assume $n>1$, which means that either there is a $2$-dimensional component $\Sigma_i$ in the singular set, or there is a singular point $q_j$ of non-Du Val type. 
\end{itemize}

With the preceding understood, we recall a construction from \cite{C1}, Theorem 1.5, that is, for any symplectic $4$-orbifold, one can canonically associate it with a symplectic $4$-manifold, called the {\it symplectic resolution}. In the present case, the symplectic resolution of $X$, to be denoted by $\tilde{X}$, is obtained as follows. First, one de-singularizes the symplectic structure on $X$ along the $2$-dimensional singular components $\{\Sigma_i\}$, which results a natural symplectic structure on
the underlying space $|X|$, making it a symplectic $4$-orbifold with only isolated singular points 
$\{q_j\}$. Each $\Sigma_i$ descends to an embedded symplectic surface in $|X|$, which will be denoted by $B_i$. With this understood, the symplectic resolution $\tilde{X}$ is simply the minimal symplectic resolution of the orbifold $|X|$. We refer the readers to \cite{C1} for more details. 
(Compare also \cite{MR}.)

For each $j$,  let $\{F_{j,k}|k\in I_j\}$ be the exceptional set in the minimal resolution of $q_j$, and 
denote by $D_j:=\cup_{k\in I_j} F_{j,k}$ be the configuration of symplectic spheres in $\tilde{X}$. 
Furthermore, we denote by $D$ the pre-image of the singular set of $X$ in $\tilde{X}$ under the resolution map $\tilde{X}\rightarrow X$. Then clearly, 
$$
D=\cup_i B_i \cup \cup_j D_j. 
$$
With this understood, we note that the assumption $c_1(K_X)=0$ implies that the canonical class
of $\tilde{X}$ is supported in $D\subset \tilde{X}$. Indeed, by Proposition 3.2 of \cite{C1},
$c_1(K_X)=0$ implies
$$
c_1(K_{\tilde{X}})=-\sum_i \frac{m_i-1}{m_i} B_i +\sum_j \sum_{k\in I_j} a_{j,k} F_{j,k},
$$
where $\{a_{j,k}\}$ is a set of rational numbers uniquely determined by the following equations:
for each $j$, we set $c_1(D_j):=\sum_{k\in I_j} a_{j,k} F_{j,k}$, then
$$
c_1(D_j)\cdot F_{j,l}+F_{j,l}^2+2=0, \;\; \forall l\in I_j. 
$$
We remark that $a_{j,k}\leq 0$, and for each $j$, $a_{j,k}=0$ for all $k\in I_j$ if and only if $m_j=1$,
i.e., $q_j$ is a Du Val singularity. 

The following fact is fundamental to the considerations in this paper.

\vspace{2mm}

{\bf Proposition 1.0.}\hspace{2mm}
{\it The resolution $\tilde{X}$ is a rational or ruled $4$-manifold.}

\vspace{2mm}

This is an immediate consequence of the assumption that $n:=\text{lcm}\{m_i,m_j\}>1$; indeed,
$n>1$ implies easily that 
$$
c_1(K_{\tilde{X}})=-\sum_i \frac{m_i-1}{m_i} B_i +\sum_j \sum_{k\in I_j} a_{j,k} F_{j,k}\neq 0,
$$
and moreover, if $\tilde{\omega}$ is the symplectic structure on $\tilde{X}$, then 
$c_1(K_{\tilde{X}})\cdot [\tilde{\omega}]<0$, which implies that $\tilde{X}$ is rational or ruled. 
(Compare also \cite{C1}, Lemma 4.1, for the case of global quotients.)

Now we state the main results of this paper. 

\begin{theorem}
There exists a symplectic $4$-orbifold $Y$ with a cyclic symplectic orbifold covering $\pi: Y\rightarrow X$ of degree $n:=\text{lcm}\{m_i,m_j\}$, which has the following properties.
\begin{itemize}
\item [{(1)}] The orbifold $Y$ has at most Du Val singularities, which are given by the set $\pi^{-1}(\{q_j|H_j\neq \{1\}\})$.
\item [{(2)}] The canonical line bundle $K_Y$ is trivial as an orbifold complex line bundle. Moreover, there exists a nowhere vanishing
section $s$ of $K_Y$ such that the induced $\Z_n$-action on $K_Y$ by the deck transformations is given by the multiplication of
$\exp(2\pi i/n)$, i.e., $s\mapsto \exp(2\pi i/n)\cdot s$, for some generator of $\Z_n$.
\item [{(3)}] The symplectic $\Z_n$-action on $Y$ by the deck transformations has the following fixed-point set structure: for each $i$,
every component in $\pi^{-1}(\Sigma_i)$ is fixed by an element of order $m_i$ in $\Z_n$, and for each $j$ with $m_j>1$, every point
in $\pi^{-1}(q_j)$ is fixed by an element of order $m_j$ in $\Z_n$. The number of components in $\pi^{-1}(\Sigma_i)$ is $n/m_i$ and
the number of points in $\pi^{-1}(q_j)$ is $n/m_j$, for each $i,j$. 
\end{itemize}
\end{theorem}

The construction of $\pi:Y\rightarrow X$ is a standard affair in the algebraic geometry setting (see e.g. \cite{BP}). Our construction may be regarded as a topological version of it. Note that even if $X$ arises as a global quotient $M/G$ where $M$ is a symplectic Calabi-Yau $4$-manifold, $Y$ is not necessarily the same as $M$; in fact, $Y\neq M$ as long as $X=M/G$ has an isolated singular point $q_j$ with $H_j$ nontrivial, e.g., a Du Val singularity, as in this case $Y$ is singular. Finally, the quotient $Y/\Z_n$ is naturally a smooth $4$-orbifold (cf. \cite{C1}, Lemma 2.1), and $Y/\Z_n=X$ as
orbifolds. 

Let $\tilde{Y}$ be the (minimal) symplectic resolution of $Y$. Then $\tilde{Y}$ is a symplectic Calabi-Yau $4$-manifold, and furthermore, the symplectic $\Z_n$-action on $Y$ naturally extends to a symplectic $\Z_n$-action on $\tilde{Y}$ (cf. \cite{C1}, Theorem 1.5(3)). We note that $\tilde{Y}$ only depends on the partial resolution $\tilde{X}^0$ of $X$, i.e., the symplectic $4$-orbifold obtained by only resolving the Du Val singularities of $X$. It is easy to see that the quotient orbifold 
$\tilde{Y}/\Z_n$ equals $\tilde{X}^0$ if and only if for each $j$ with $m_j>1$, the subgroup $H_j$ is trivial. Finally, note that $\tilde{Y}=Y$ (so that $\tilde{Y}/\Z_n=X$) if and only if $H_j$ is trivial for each $j$. We shall call $Y$ the {\it Calabi-Yau cover} of $X$. 

The implication of Theorem 1.1 is two-fold. On the one hand, it gives us a way to construct symplectic Calabi-Yau $4$-manifolds; we shall explore this in a future occasion. On the other 
hand, by exploiting the $\Z_n$-action on $\tilde{Y}$, we may obtain information on the singular set of $X$. Building on earlier work \cite{C}, we are led to

\begin{theorem}
The Calabi-Yau cover $Y$ and its symplectic resolution $\tilde{Y}$ are classified according to the topology of $X$ as follows:
\begin{itemize}
\item [{(1)}] Suppose $b_1(X)>0$. Then the singular set of $X$ consists of only tori with self-intersection zero. In this case, $Y=\tilde{Y}$, which is a $T^2$-bundle over $T^2$ with symplectic fibers. 
\item [{(2)}] Suppose $b_1(X)=0$. Then $\tilde{Y}$ is an integral homology 
$K3$ surface, unless $X$ falls into one of the following
two cases: $\left(i\right)$ the singular set of $X$ consists of $9$ non-Du Val isolated points 
of isotropy of order $3$, or $\left(ii\right)$ the singular set of $X$ consists of $5$ isolated points of isotropy of order $5$ which are all of type $(1,2)$. In both cases $\left(i\right)$ and $\left(ii\right)$, $Y=\tilde{Y}$, which is a rational homology $T^4$. 
\end{itemize}
\end{theorem}

We remark that, as a consequence of Theorem 1.2, it remains to further investigate the topology of $X$ for the case where $b_1(X)=0$, as far as the topology of $\tilde{Y}$ is concerned (if $X=M/G$ is
a global quotient and $b_1(X)>0$, then $M$ must be a $T^2$-bundle over $T^2$, cf. \cite{C}, Theorem 1.1). The following problems are fundamental.

\begin{problem}
Suppose $b_1(X)=0$.
\begin{itemize}
\item[{(1)}] Determine the singular set, i.e., the possible topological type, including the orders of
isotropy of the $2$-dimensional singular components, of $X$.
\item[{(2)}] Determine whether $X$ admits a complex structure.
\item[{(3)}] Determine whether a possible topological type of the singular set of $X$ can be realized.
\item[{(4)}] Determine the orbifold fundamental group of $X$ (as well as that of the partial resolution 
$\tilde{X}^0$, i.e., the symplectic $4$-orbifold obtained by only resolving the Du Val singularities of $X$).
\end{itemize}
\end{problem}

Concerning Problem 1.3(1), it remains to consider the case where $\tilde{Y}$ is an integral homology $K3$ surface. It is conceivable that there are only finitely many possible topological types for the singular set of $X$. 

Concerning Problem 1.3(2)-(4), our strategy is to consider the resolution $\tilde{X}$ of $X$,
which is a symplectic rational $4$-manifold, and to consider the embedding $D\subset \tilde{X}$, 
which is a disjoint union of configurations of symplectic surfaces in $\tilde{X}$. The relevant questions concerning $X$ are then reduced to corresponding questions concerning the embedding of $D\subset \tilde{X}$. For instance, Problem 1.3(2) is equivalent to the question as whether the embedding of $D$ in $\tilde{X}$ can be made holomorphic. 

In order to study the embedding $D\subset \tilde{X}$, our strategy is to first determine the homology classes of the components of $D$, i.e., the symplectic surfaces $B_i$, $F_{j,k}$, with respect to
a certain standard basis of $H^2(\tilde{X})$ (called a reduced basis, which depends on the choice of symplectic structure on $\tilde{X}$). The necessary technical machinery was developed in \cite{C}; in particular, it is shown that there are essentially only finitely many possible homological expressions of $D$ in $\tilde{X}$. To go beyond the homological classification of $D$, we introduce in this paper a general successive blowing down procedure, to be applied to a symplectic rational $4$-manifold containing a configuration of symplectic surfaces (we remark that a simple version of this procedure was already employed in \cite{C} in the proof of its main result). In particular, this successive blowing down procedure allows us to reduce $\tilde{X}$ to either $\C\P^2$ or $\C\P^2\#\overline{\C\P^2}$, and to construct a canonical descendant of $D$ in $\C\P^2$ or $\C\P^2\#\overline{\C\P^2}$. This descendant of $D$, to be denoted by $\hat{D}$, is a union of pseudoholomorphic curves with controlled singularities and intersection properties, which depend only on the homological expression of $D$. The procedure is reversible: by successively blowing up $\hat{D}$, we can recover $D\subset \tilde{X}$ up to a smooth isotopy. In this way, we reduce the questions concerning the embedding $D\subset \tilde{X}$ to relevant questions concerning the embedding of $\hat{D}$ in $\C\P^2$ or $\C\P^2\#\overline{\C\P^2}$, which we believe are more amenable to the current existing techniques in symplectic topology. 

\vspace{2mm}

The organization of this paper is as follows. In Section 2, we prove Theorems 1.1 and 1.2.
In addition, for the purpose of illustration we also include at the end of the section a few examples of the orbifold $X$. These examples arise as the quotient orbifold of a holomorphic action on a hyperelliptic surface or a complex torus, and their singular sets do not belong to the case (i) or (ii) in Theorem 1.2. The corresponding symplectic Calabi-Yau $4$-manifold $\tilde{Y}$ is a $K3$ surface,
equipped with a non-symplectic automorphism of finite order. A common feature of these examples is that the $K3$ surface contains a large number of $(-2)$-curves. Section 3 contains some general constraints on the singular set of $X$. There are two, seemingly independent, sources for the constraints. One type of the constraints is obtained by analyzing the symplectic $\Z_n$-action on
the Calabi-Yau homology $K3$ surface $\tilde{Y}$, while the other type is derived from the Seiberg-Witten-Taubes theory. In Section 4, we give a detailed account of the successive blowing down procedure. Furthermore, at the end of the section we apply the procedure to a few concrete examples for the purpose of illustration. 

\vspace{3mm}

{\bf Acknowledgement:} We thank Alan L. Edmonds, Paul Hacking, and Weiwei Wu for useful communications. We are also grateful to an anonymous referee for a critical reading of the paper.

\section{Proof of the main theorems}

Recall that an orbifold complex line bundle $p: L\rightarrow X$ is said to be trivial if there is a collection of local trivializations of $L$ such that the associated transition functions are given by identity maps,
and moreover, for any local trivialization of $L$ over an uniformizing system $(U,G)$, the $G$-action on $p^{-1}(L|_U)\cong U\times \C$ is trivial on the $\C$-factor. Note that this latter condition is equivalent to the statement that $L$ descends to an ordinary complex line bundle over the underlying topological space $|X|$. It follows easily that if $L$ is a trivial orbifold complex line bundle, then the underlying total space of $L$, denoted by $|L|$, is given by the product $|X|\times \C$. 

\begin{lemma}
Set $L:=K_X$. Then the $n$-th tensor power $L^n$ is a trivial orbifold complex line bundle over $X$. Moreover, $n$ is the minimal positive integer having this property, i.e., if $L^m$ is a trivial orbifold complex line bundle for some $m>0$, then $m$ must be divisible by $n$. 
\end{lemma}

\begin{proof}
For each $i$, let $\nu_i:=TX/T\Sigma_i$ be the normal bundle of $\Sigma_i$ in $X$, and let $U_i$
be the associated disc bundle of $\nu_i$. Then there is a natural smooth $\Z_{m_i}$-action on
$U_i$, fixing the zero section and free in its complement, such that $U_i/\Z_{m_i}$ is identified with a regular neighborhood of $\Sigma_i$ in $|X|$. (We may regard $(U_i,\Z_{m_i})$ as an uniformizing 
system of $X$ along $\Sigma_i$.)
Note that $H^2(U_i)$ is torsion-free, so that $c_1(K_X)=0$ in $H^2(|X|;\Q)$ implies that $K_{U_i}$ is trivial. With this understood, it is easy to see that there is a trivialization $K_{U_i}\cong U_i\times \C$ such that the induced $\Z_{m_i}$-action on the trivialization 
is given by the multiplication of $\exp(2\pi i/m_i)$ on the $\C$-factor for some generator of $\Z_{m_i}$. On the other hand, for each $j$, if we let $(U_j,G_j)$ be an uniformizing system centered at $q_j$, where $U_j$ is a $4$-ball, then $K_{U_j}\cong U_j\times \C$, and the induced $G_j$-action is given by the multiplication of $\exp(2\pi i/m_j)$ for a generator of $G_j/H_j$. With this understood, we see immediately that $L^n$ descends to an ordinary complex line bundle over $|X|$. Furthermore, it also follows easily that if $L^m$ is a trivial orbifold complex line bundle for some $m>0$, then $m$ must be divisible by $n=\text{lcm}\{m_i,m_j\}$.

To show that $L^n$ is the trivial orbifold complex line bundle, it remains to prove that $L^n$ descends to a trivial ordinary complex line bundle over $|X|$. With this understood, we note that $c_1(L^n)=nc_1(L)=0$ in $H^2(|X|;\Q)$, and with $L^n$ as  an ordinary complex line bundle over 
$|X|$, $c_1(L^n)$ admits a lift in $H^2(|X|)$ (still denoted by $c_1(L^n)$ for simplicity), which is torsion. The assertion that $L^n$ descends to a trivial ordinary complex line bundle over $|X|$ follows readily from the claim that $c_1(L^n)=0$ in $H^2(|X|)$. 

We shall prove that $H^2(|X|)$ is torsion-free, so that $c_1(L^n)=0$ in $H^2(|X|)$ as claimed. To see this, we note that the symplectic resolution $\tilde{X}$ of $X$ is either rational or ruled. Moreover, note that $\pi_1(|X|)=\pi_1(\tilde{X})$, where $\pi_1(\tilde{X})=0$ when $\tilde{X}$ is rational, and 
$\pi_1(\tilde{X})=\pi_1(\Sigma)$ when $\tilde{X}$ is a ruled surface over a Riemann surface $\Sigma$. In any event, $H_1(|X|)$ is torsion-free, so that $H^2(|X|)=Hom(H_2(|X|),\Z)$, which implies that $H^2(|X|)$ is torsion-free as well. This finishes the proof of the lemma.

\end{proof}

\noindent{\bf Proof of Theorem 1.1:}

\vspace{2mm}

Let $t$ denote the tautological section of the pull-back bundle of $p: L\rightarrow X$ over the total space $L$, i.e., for each $x\in L$, $t(x)=x\in (p^\ast L)_x=L_{p(x)}$. Then consider $\xi:=t^n$, the 
$n$-th tensor power of $t$, which is a section of the pull-back bundle of $L^n$ over $X$ to the total space $L$. Since $L^n$ is trivial (as orbifold complex line bundle), we can fix a trivialization 
$|L|\times \C$ of the pull-back bundle $p^\ast L^n\rightarrow L$, and denote by $1$ the constant section $|L|\times \{1\}$. With this understood, we set $Y:=\xi^{-1}(1)$, as a subset of the total space $L$. The map $\pi: Y\rightarrow X$ is simply given by the restriction of $p: L\rightarrow X$ to $Y$. Let 
$\lambda$ be the generator of $\Z_n$ which acts on $L$ by fiber-wise complex multiplication by 
$\exp(2\pi i/n)$. Then it is clear that the tautological section $t$ is equivariant under the $\Z_n$-action, i.e., $t(\lambda\cdot x)=\lambda\cdot t(x)$. With this understood, note that $\xi(\lambda\cdot x)=\xi(x)$, which implies that the set $Y$ is invariant under the action of $\lambda$. Furthermore, note that
the quotient space of $Y$ under the $\Z_n$-action is identified with $X$ under $\pi: Y\rightarrow X$. 

With the preceding understood, we shall first show that $Y$ is a smooth orbifold and $\pi: Y\rightarrow X$ is a smooth orbifold covering. Equipping $Y$ with the pull-back symplectic structure, 
$\pi: Y\rightarrow X$ becomes a symplectic orbifold covering. 

To see that $Y$ is a smooth orbifold, we note that the tautological section $t$ is given by an equivariant section for any local trivialization of the pull-back of $L$ over an uniformizing system, and the argument we give below is obviously equivariant. With this understood, let $v$ be any given direction along the fiber of $L^n$. Suppose $x\in \xi^{-1}(1)=Y$ be any point. 
We choose a direction $u$ along the fiber of $L$ such that $ux^{n-1}=\frac{1}{n}v$ holds as 
tensor product (this is possible because $x\neq 0$ in $L$). Then it is easy to check that 
$$
\frac{d}{ds}(t^n(x+su))|_{s=0}=nux^{n-1}=v,
$$
which implies that the section $\xi$ intersects the constant section $1$ transversely. It follows that $Y$ is a smooth orbifold, which is easily seen of dimension $4$. 

Next we show that $\pi: Y\rightarrow X$ is a smooth orbifold covering. We shall only be inspecting the situation near the singular set of $X$, as the matter is trivial over the smooth locus. To this end, we first consider the singular components $\Sigma_i$. 
We continue to use the notation from Lemma 2.1, where
$U_i$ denotes the disc bundle of the normal bundle of $\Sigma_i$ in $X$, with a natural $\Z_{m_i}$-action on $U_i$ such that $U_i/\Z_{m_i}$ gives a regular neighborhood of $\Sigma_i$ in $|X|$. 
As we have seen before, $K_{U_i}$ is trivial, so we can fix a trivialization $K_{U_i}\cong U_i\times \C$.
Let $\delta_i\in \Z_{m_i}$ be the generator such that the action of $\delta_i$ on $U_i\times \C$ is given by the multiplication of $\exp(2\pi i/m_i)$ for the $\C$-factor. With this understood, note that 
$(U_i\times \C,\Z_{m_i})$ is an uniformizing system for the orbifold $L$, over which the pull-back bundle $p^\ast L\rightarrow L$ admits a natural trivialization $(U_i\times \C)\times \C\rightarrow
U_i\times \C$, where $\delta_i$ also acts as multiplication by $\exp(2\pi i/m_i)$ on the last $\C$-factor. With this understood, we note that $Y=\xi^{-1}(1)$ is given, in the uniformizing system 
$(U_i\times \C,\Z_{m_i})$, by the subset 
$$
V_i:=\{(y, z)\in U_i\times \C| y\in U_i, z^n=1\}.
$$
Furthermore, the action of $\lambda\in\Z_n$ is given by $(y,z)\mapsto (y,\exp(2\pi i/n)z)$, and the action of $\delta_i$ is given by $(y,z)\mapsto (\delta_i\cdot y,\exp(2\pi i/m_i)z)$. It follows easily that the quotient space of $V_i=\{(y, z)\in U_i\times \C| y\in U_i, z^n=1\}$ under the $\Z_{m_i}$-action can be identified with 
$$
\{(y, z)\in U_i\times \C| y\in U_i, z^{n/m_i}=1\},
$$
which is a disjoint union of $n/m_i$
many copies of $U_i$. This shows easily that over $\pi^{-1}(U_i/\Z_{m_i})$, $Y$ is smooth, 
$\pi: Y\rightarrow X$ is given by the quotient map of the action of $\delta_i^{-1} \lambda^{n/m_i}$,
and the number of components in $\pi^{-1}(\Sigma_i)$ is $n/m_i$. In particular, $\pi: Y\rightarrow X$ is a smooth orbifold covering near each $\Sigma_i$. 

The situation near each $q_j$ is similar. If we let $(U_j,G_j)$ be the uniformizing system near $q_j$, then in $(U_j\times \C,G_j)$, $Y$ is given by the subset 
$$
V_j:=\{(y, z)\in U_j\times \C| y\in U_j, z^n=1\}.
$$ 
Moreover, the quotient space by the $G_j$-action can be identified with 
$$
\{([y], z)\in (U_j/H_j)\times \C| [y]\in U_j/H_j, z^{n/m_j}=1\},
$$
which is a disjoint union of $n/m_j$ many copies of $U_j/H_j$. It follows easily that 
over each component of $\pi^{-1}(U_j/G_j)$, $Y$ is a smooth orbifold, uniformized by $(U_j,H_j)$. Furthermore, 
$\pi: Y\rightarrow X$ is a smooth orbifold covering near each $q_j$, and the number of points in
$\pi^{-1}(q_j)$ is $n/m_j$. Note that in particular, the argument above proved part (1) and part (3) of Theorem 1.1.

Since we endow $Y$ with the pull-back symplectic structure via $\pi: Y\rightarrow X$, one has $K_Y=\pi^\ast K_X=\pi^\ast L=(p^\ast L)|_Y$, as $\pi=p|_Y$. With this understood, the restriction of the tautological section $s:=t|_Y$ is a nowhere vanishing section of $K_Y$. Moreover, the action of 
$\lambda\in \Z_n$ is given by $s\mapsto \exp(2\pi i/n) s$. This proves part (2). 

It remains to show that $Y$ is connected, which is a consequence of $n=\text{lcm}\{m_i,m_j\}$ being minimal. To see this, suppose $Y$ is not connected, and let $Y_0$ be a connected component of $Y$. Then there is a factor $m>1$ of $n$ such that $\lambda^m$ generates the subgroup of $\Z_n$ which leaves $Y_0$ invariant. With this understood, note that the action of $\lambda^m$ is
trivial on $K_{Y_0}^{n/m}=K_Y^{n/m}|_{Y_0}=(\pi^\ast L)^{n/m}|_{Y_0}$. Since 
$(\pi^\ast L)^{n/m}|_{Y_0}$ is trivial, it follows that $L^{n/m}$ must be the trivial orbifold complex line bundle over $X$. But this contradicts Lemma 2.1, hence $Y$ is connected. The proof of Theorem 1.1 is complete. 

\vspace{2mm}

Next we prove Theorem 1.2. The proof relies heavily on the fixed-point set analysis of symplectic finite cyclic actions on symplectic Calabi-Yau $4$-manifolds with $b_1>0$ carried out in \cite{C}. 
Furthermore, some standard facts about the topology of symplectic Calabi-Yau $4$-manifolds will also be used in the proof.

We shall first have a recollection of these facts, see \cite{Bauer, Li, Li1} for more details. 
Let $M$ be a symplectic Calabi-Yau $4$-manifold. Then $M$ is either an integral homology $K3$ surface or a rational homology $T^2$-bundle over $T^2$. The latter case corresponds to $M$ having positive first Betti number, and we should note the following facts about this case, which are used frequently:
$$
\chi(M)=Sign(M)=0, \;\;  b_2^{+}(M)=b_2^{-}(M)=b_1(M)-1, \mbox{ and } 2\leq b_1(M)\leq 4.
$$
Moreover, if $b_1(M)=4$, then $M$ must be a rational homology $T^4$. 

\vspace{2mm}

\noindent{\bf Proof of Theorem 1.2:}

\vspace{2mm}

We begin by noting that the two cases $b_1(X)>0$ and $b_1(X)=0$ correspond to the resolution
$\tilde{X}$ being irrational ruled and rational respectively. On the other hand, since we will consider the symplectic $\Z_n$-action on $\tilde{Y}$, it is useful to note that, as $Y/\Z_n=X$ as orbifolds, the resolution of the quotient orbifold $\tilde{Y}/\Z_n$ is in the same symplectic birational equivalence class with $\tilde{X}$ (cf. \cite{C1}, Theorem 1.5(3)).

\vspace{2mm}

{\bf Case (1): $b_1(X)>0$.} In this case, the resolution of $\tilde{Y}/\Z_n$ is irrational ruled as we noted 
above. By Theorem 1.1 of \cite{C}, $\tilde{Y}$ is a $T^2$-bundle over $T^2$, and moreover, from 
its proof in \cite{C}, we also know that the fibers of the $T^2$-bundle are symplectic. 

To see that $Y=\tilde{Y}$, we note that $Y$ must be smooth. This is because if $Y$ has a singular
point, then its minimal resolution in $\tilde{Y}$ gives a configuration of symplectic spheres in
$\tilde{Y}$, contradicting the fact that $\pi_2(\tilde{Y})=0$ (as $\tilde{Y}$ is a $T^2$-bundle over 
$T^2$). This proves that $Y=\tilde{Y}$. Note
that as a consequence of $Y$ being smooth, $X$ has no singular point $q_j$ with $H_j$ nontrivial
(cf. Theorem 1.1(1)). With this understood, the proof of Theorem 1.2(1) is completed by the following
lemma.

\begin{lemma}
There is no isolated singular point $q_j$ of $X$ with $m_j>1$, and any $2$-dimensional singular component of $X$ is a torus of self-intersection zero.
\end{lemma}

\begin{proof}
Suppose to the contrary that there is a $q_j$ with $m_j>1$. We pick a 
$\tilde{q}_j\in \pi^{-1}(q_j)\subset Y$, which, by Theorem 1.1, is fixed by an element of $\Z_n$
of order $m_j$. We let $H$ be a subgroup of $\Z_n$ of prime order which fixes $\tilde{q}_j$.
Then since $H_j$ is trivial, it is easy to see, from the proof of Theorem 1.1, 
that the image of $\tilde{q}_j$ in $Y/H$ is a 
non-Du Val singularity. By Lemma 4.1 of \cite{C1}, the resolution of $Y/H$ is either rational or ruled. 
We claim that the resolution of $Y/H$ must be irrational ruled. To see this, note that 
the resolution of $Y/\Z_n$ and the orbifold $Y/\Z_n$ itself have the same
first Betti number. Since the resolution of $Y/Z_n$ (which is $\tilde{Y}/\Z_n$ as $Y=\tilde{Y}$) is irrational ruled, it follows that $b_1(Y/\Z_n)>0$, which implies $b_1(Y/H)\geq b_1(Y/\Z_n)>0$.
This shows that the resolution of $Y/H$ is irrational ruled because it has the same first Betti
number with $Y/H$. By Theorem 1.2(2) of \cite{C}, the resolution of $Y/H$ being irrational ruled 
implies that the action of $H$ on $Y$ has only $2$-dimensional fixed components, which are 
tori of self-intersection zero. But this is a contradiction as $\tilde{q}_j$ is an isolated fixed
point of $H$. 

By the same argument, any singular component $\Sigma_i$ of $X$ must be a torus of self-intersection zero. This finishes the proof of the lemma.

\end{proof}

{\bf Case (2): $b_1(X)=0$.} In this case, $\tilde{X}$ is rational, so is the resolution of $\tilde{Y}/\Z_n$.

We first assume $\tilde{Y}$ is a symplectic Calabi-Yau $4$-manifold with $b_1>0$. We shall prove 
that $Y=\tilde{Y}$ and $b_1(\tilde{Y})=4$, and determine the singular set of $X$.

\begin{lemma}
There are no singular points in $Y$. Moreover, $b_1(\tilde{Y})\neq 3$.
\end{lemma}

\begin{proof}
Note that $2\leq b_1(\tilde{Y})\leq 4$. We shall first consider the case where $b_1(\tilde{Y})=4$.
In this case, $\tilde{Y}$ is a rational homology $T^4$. To see that $Y$ is smooth, we recall
the following fact from \cite{Sr}, that is, the cohomology ring $H^\ast(\tilde{Y};\R)$ is isomorphic 
to the cohomology ring $H^\ast(T^4;\R)$. A consequence of this is that the Hurwitz map $\pi_2(\tilde{Y})\rightarrow H_2(\tilde{Y})$ has trivial image. If $Y$ has singularities, then the exceptional set of their resolutions in $\tilde{Y}$ consists of symplectic $(-2)$-spheres, which is a contradiction. The lemma 
is proved for the case where $b_1(\tilde{Y})=4$.

Next, consider the case where $b_1(\tilde{Y})=3$. For this case, we recall Lemma 2.6 in \cite{C}
which says that under the condition $b_1(\tilde{Y})=3$, if the resolution of $\tilde{Y}/\Z_n$ is
rational or ruled, then the $\Z_n$-action must be an involution. On the other hand, later in the proof
of Theorem 1.1 of \cite{C}, the case where the resolution of $\tilde{Y}/\Z_n$ is rational is actually
eliminated, as it was shown in this case that $\tilde{Y}$ must be diffeomorphic to a hyperelliptic 
surface (in particular, $b_1(\tilde{Y})=2$ which is a contradiction). Hence $b_1(\tilde{Y})\neq 3$.

Finally, assume $b_1(\tilde{Y})=2$. In this case, $b_2^{-}(\tilde{Y})=1$, which implies easily 
that $Y$ can have at most one singular point. Let $\tilde{q}_j\in Y$ be such a singular point and
let $q_j=\pi(\tilde{q}_j)$ be the singular point in $X$. Note that $b_2^{-}(\tilde{Y})=1$ implies that
the exceptional set of $\tilde{q}_j$ in the minimal resolution in $\tilde{Y}$ consists of a single 
symplectic $(-2)$-sphere, which we denote by $E$. Furthermore, it is easy to see that $H_j=\Z_2$
and $n/m_j=1$, as $\pi^{-1}(q_j)$ consists of only one point $\tilde{q}_j$. To derive a contradiction,
we let $H$ be any prime order subgroup of $\Z_n$ which acts on $Y$ and $\tilde{Y}$, and let $p$ be the
order of $H$. Note that $H$ fixes the point $\tilde{q}_j\in Y$, so the $H$-action on $\tilde{Y}$
leave the $(-2)$-sphere $E$ invariant. Finally, according to the fixed-point set analysis of prime order
symplectic actions given in Theorem 1.2 of \cite{C}, the order $p$ must be $2$ or $3$. Furthermore,
the action of $H$ on $\tilde{Y}$ can only have tori of self-intersection zero as fixed components. 

We claim that $H$ must fix the $(-2)$-sphere $E$, which is a contradiction to the classification of the fixed components 
of $H$ mentioned above. To see this, note that from the proof of Theorem 1.1,
it is easy to see that the isotropy subgroup of $\tilde{q}_j\in Y$ of the $\Z_n$-action on $Y$ can be
naturally identified with $G_j/H_j$. Since $\tilde{q}_j$ is fixed by $H$, we may regard $H$ as
a subgroup of $G_j/H_j$. Let $H^\prime\subset G_j$ be the pre-image of $H$ under 
$G_j\rightarrow G_j/H_j$. Then it follows easily from the fact that $H_j=G_j\cap SU(2)$, $H_j=\Z_2$
and $H$ has order $2$ or $3$, that $H^\prime$ must be a cyclic group of order $4$ or $6$.
Furthermore, the action of $H^\prime$ (as a subgroup of $G_j$) on the uniformizing system at
$q_j$ has weights $(1,1)$. It follows immediately that the $(-2)$-sphere $E$ in $\tilde{Y}$
is fixed under the $H$-action. This finishes the proof of the lemma.

\end{proof}

Since $Y$ is nonsingular by Lemma 2.3, we see immediately that $Y=\tilde{Y}$. 
We shall next prove $b_1(Y)=4$ and $n=3$ or $5$. But first, note that
Lemma 2.3 implies that for any singular point $q_j$ of $X$, the group $H_j$ is trivial. Consequently,
for any prime order subgroup $H$ of $\Z_n$ acting on $Y$, the orbifold $Y/H$ does not have any
Du Val singularity. On the other hand, since $n=\text{lcm}\{m_i,m_j\}$, the order $p$ of $H$ must
be a factor of one of $m_i$ or $m_j$. It follows from Theorem 1.1(3) that the $H$-action on $Y$
must fix either a component in $\pi^{-1}(\Sigma_i)$ for some $i$, or a point in $\pi^{-1}(q_j)$ for
some $j$; in particular, the action of $H$ on $Y$ is not free. By Lemma 4.1 of \cite{C1}, the
resolution of $Y/H$ is either rational or ruled.

\begin{lemma}
There must be a subgroup $H$ of prime order such that the resolution of $Y/H$ is rational. 
\end{lemma}

\begin{proof} 
Suppose to the contrary that for every subgroup of prime order, the resolution of the group action is irrational ruled. Then it follows easily from Theorem 1.2(2) of \cite{C} that the orbifold $X=Y/\Z_n$ has only $2$-dimensional singular components. 

We pick a subgroup $\Gamma$ of prime order. Since $\Gamma$ has no isolated fixed points, the resolution of $Y/\Gamma$ is simply the underlying space $|Y/\Gamma|$, which is a $\s^2$-bundle over $T^2$ (cf. \cite{C}, Theorem 1.2(2)). 
Moreover, the fixed-point set of $\Gamma$ consists of tori of self-intersection zero 
whose images in $|Y/\Gamma|$ intersect transversely with the fibers of the $\s^2$-bundle
(see the proof of Theorem 1.1 in \cite{C}). In fact, more is proved in \cite{C}, i.e., for any compatible
almost complex structure $J$ on $|Y/\Gamma|$ which is integral near the fixed-point set of 
$\Gamma$, the $\s^2$-bundle on $|Y/\Gamma|$ can be chosen to have $J$-holomorphic fibers.
With this understood, we consider the induced action of $\Z_n/\Gamma$ on $|Y/\Gamma|$,
which can be made symplectic (cf. \cite{C1}). It is clear that the fixed-point set of $\Gamma$ is
invariant under $\Z_n/\Gamma$, so that we may choose $J$ to be $\Z_n/\Gamma$-invariant. 

To derive a contradiction, note that the $\Z_n/\Gamma$-action on $|Y/\Gamma|$ preserves the
$\s^2$-bundle structure on $|Y/\Gamma|$, so that there is an induced action of $\Z_n/\Gamma$
on the base of the $\s^2$-bundle, which is $T^2$. Now since $X=Y/\Z_n$ and $b_1(X)=0$, 
it follows easily that the $\Z_n/\Gamma$-action on $|Y/\Gamma|$ must induce a homologically nontrivial action on the base $T^2$. This implies that there must be an element 
$g\in \Z_n/\Gamma$ which leaves a $\s^2$-fiber invariant. If $g$ fixes the $\s^2$-fiber, then $X$
would have $2$-dimensional singular components that are not disjoint, because the $\s^2$-fiber
fixed by $g$ intersects with the fixed-point set of $\Gamma$. This is a contradiction, so $g$ must
fix two isolated points on the $\s^2$-fiber. It is clear that these two isolated fixed points of $g$ 
can not be contained in the fixed-point set of $\Gamma$, hence give two isolated singular 
points of $X$. This is also a contradiction as $X$ has only $2$-dimensional singular components. 
Hence the claim that there must be a subgroup $H$ of prime order such that the resolution 
of $Y/H$ is rational. This proves the lemma.

\end{proof}

An immediate consequence of Lemma 2.4 is that $b_1(Y)=4$. This is because if $b_1(Y)\neq 4$,
then $b_1(Y)=2$ by Lemma 2.3. In this case, let $H$ be a subgroup of $\Z_n$ of prime order such
that the resolution of $Y/H$ is rational. Then by the classification of fixed-point set structures 
in Theorem 1.2(3) of \cite{C}, $Y/H$ has Du Val singularities. But this is a contradiction to 
Lemma 2.3, hence $b_1(Y)=4$. 

\begin{lemma}
There is only one subgroup $H$ of prime order such that the resolution of $Y/H$ is rational. 
Moreover, the order of $H$ is either $3$ or $5$, and $n$ does not have any prime factor which is
not equal to the order of $H$. 
\end{lemma}

\begin{proof}
By Theorem 1.2(3) of \cite{C}, if $H$ is a prime order subgroup such that the resolution of $Y/H$ is rational, then the order of $H$ must be either $3$ or $5$. Moreover, if $H$ has order $3$, the
fixed-point set $Y^H$ consists of $9$ isolated points, and if $H$ has order $5$, $Y^H$ consists 
of $5$ isolated points. With this understood, we note that since $\Z_n$ is cyclic, there is a unique
subgroup of order $p$ for each prime factor $p$ of $n$.

Suppose to the contrary that there are prime order subgroups $H_1,H_2$ of order $3$, $5$
respectively. Then clearly, $Y^{H_1}$ is invariant under $H_2$, and 
$Y^{H_2}$ is invariant under $H_1$. Examining the action of $H_2$ on $Y^{H_1}$, which
consists of $9$ points, it is easy to see that $H_2$ must fix exactly $4$ points in $Y^{H_1}$. 
So $Y^{H_1}\cap Y^{H_2}$ consists of $4$ points. On the other hand, examining the action 
of $H_1$ on $Y^{H_2}$, it either fixes $2$ points or the entire set. This implies that 
$Y^{H_1}\cap Y^{H_2}$ either consists of $2$ points or $5$ points, a contradiction.
Hence the claim that there 
is only one subgroup $H$ of prime order such that the resolution of $Y/H$ is rational, and that
the order of $H$ is either $3$ or $5$.

It remains to show that $n$ does not have any prime factor which is not equal to the order of $H$. 
To see this, suppose to the contrary that there is a prime factor $p$ of $n$, which is not equal to
the order of $H$. Then there is a subgroup $\Gamma$ whose order equals $p$. Note that 
 the resolution of $Y/\Gamma$ must be irrational ruled, hence by Theorem 1.2(2), $p=2$ or $3$. 
 Moreover, the fixed-point set $Y^\Gamma$ consists of a disjoint union of tori of self-intersection zero,
 which is disjoint from $Y^H$. If $H$ has order $3$, then $\Gamma$ must be an involution. 
 Examining the action of $\Gamma$ on $Y^H$, which consists of $9$ points, there must be a
 point fixed by $\Gamma$, a contradiction. If $H$ has order $5$, then $p=2$ or $3$. In any event,
 $\Gamma$ must also fix a point in $Y^H$ since $Y^H$ consists of $5$ points. This shows that
 $\Gamma$ can not exist, and the lemma is proved. 
 
 \end{proof}
 
 As a consequence, $n$ must be a power of either $3$ or $5$. If $n$ is neither $3$ nor $5$, then 
 there must be an element $g\in \Z_n$ of order $9$ or $25$. Let $\theta_1,\theta_2$ be the angles 
 associated to the action of $g$ on $H^1(Y;\R)$ in Lemma 2.7 of \cite{C}. Then by Lemma 2.7 of \cite{C}, the Lefschetz number $L(g,Y)=4(1-\cos\theta_1)(1-\cos\theta_2)$, which is an integer, and 
 $2(\cos\theta_1+\cos\theta_2)\in\Z$. It is easy to check that with the order of $g$ being $9$ or $25$,
 this is not possible. Hence $n=3$ or $5$. 
Now with $n=3$ or $5$, the singular set of $X=Y/\Z_n$ must be as in (i) or (ii) of
Theorem 1.2 by the classification of fixed-point sets in Theorem 1.2(3) of \cite{C}. 

Conversely, if $X$ is given as in (i) or (ii) of Theorem 1.2, then it is clear that $n=3$ or $5$, and $Y=\tilde{Y}$. We claim that $\chi(Y)=0$. To see this, we use the Lefschetz fixed point theorem. For example, 
for the case of $n=3$, we note that the resolution $\tilde{X}=\C\P^2\# 12\overline{\C\P^2}$ because $c_1(K_{\tilde{X}})^2=-3$. This implies that $\chi(X)=\chi(\tilde{X})-9=15-9=6$. By the Lefschetz fixed point theorem, $3\chi(Y/\Z_3)=\chi(Y)+(3-1)\cdot \#Y^{\Z_3}$. With $X=Y/\Z_3$ and 
$\# Y^{\Z_3}=9$, we obtain $\chi(Y)=0$. The case of $n=5$ is similar. Hence our claim 
that $\chi(Y)=0$. It follows that $Y$ has $b_1>0$. By the classification in Theorem 1.2(3) of \cite{C},
we have $b_1(Y)=4$. This finishes the proof of Theorem 1.2. 

\begin{example}
We list a few examples of holomorphic $G$-actions on a hyperelliptic surface or complex torus $M$ such that the quotient orbifold $X=M/G$ does not have the singular set in (i) or (ii) of Theorem 1.2 
and $b_1(X)=0$. Hence by Theorem 1.2, the corresponding symplectic Calabi-Yau $4$-manifold
$\tilde{Y}$ must be a $K3$ surface, with the symplectic $\Z_n$-action defining an automorphism 
of the $K3$ surface. Note that the automorphism must be non-symplectic, because if it were
symplectic, the resolution of $\tilde{Y}/\Z_n$ must also be a $K3$ surface. However, we know that 
the resolution of $\tilde{Y}/\Z_n$ is in the same symplectic birational equivalence class with 
$\tilde{X}$ (cf. \cite{C1}, Theorem 1.5(3)), which is rational. 
An interesting feature of these examples is that the $K3$ surface contains a large number of $(-2)$-curves appearing in various types of configurations in the complement of the fixed-point set of the 
non-symplectic automorphism, coming from the resolution of the Du Val singularities in $X$.

(1) Take a holomorphic involution on a hyperelliptic surface which fixes $2$ tori and $8$ isolated points. The orbifold $X$ has a singular set of $2$ embedded tori and $8$ isolated points of 
Du Val type. The $K3$ surface $\tilde{Y}$ admits a non-symplectic involution,
which fixes $2$ tori, and in the complement of the fixed-point set, there are $16$ disjoint $(-2)$-curves.

(2) Take a holomorphic $\Z_3$-action on a hyperelliptic surface, which has $6$ isolated fixed points
and either no fixed curve or a single fixed torus (both cases are possible), where exactly $3$ of the isolated fixed points are Du Val. In this case, $\tilde{Y}$ comes with a non-symplectic automorphism of order $3$, which has $3$ isolated fixed points and either no fixed curve or a single fixed torus, such that in the complement of the fixed-point set, there are $9$ pairs of $(-2)$-curves, each intersecting transversely in one point. 

(3) Consider a holomorphic $\Z_4$-action on a hyperelliptic surface, which has $4$ isolated fixed points where $2$ of them are Du Val, and $4$ isolated points of isotropy of order $2$. The quotient orbifold $X$ has $6$ singular points, of which $4$ are Du Val. It is easy to see that $n=2$ in this example, so the $K3$ surface $\tilde{Y}$ comes with a non-symplectic involution. Note that the orbifold $Y$ has $4$ Du Val singularities of order $4$, and $6$ Du Val singularities of order $2$, where $2$ of the order $2$ singularities are fixed by the $\Z_2$ deck transformation. It follows easily that the non-symplectic involution on $\tilde{Y}$ has $2$ fixed $(-2)$-curves, and in the complement there are $4$ disjoint $(-2)$-curves and $4$ linear chains of $(-2)$-curves, each containing $3$ 
curves. (Totally, we see $18$ $(-2)$-curves in $\tilde{Y}$.) 

(4) Finally, we consider a holomorphic $\Z_8$-action on a complex torus. It has $2$ isolated fixed points, all of type $(1,5)$, $2$ isolated points of isotropy of order $4$ of type $(1,1)$, and $12$ isolated points of isotropy of order $2$. The quotient orbifold $X$ has $6$ singular points, 
of which $3$ are Du Val singularities. It is easy to see that $n=4$ in this example. Note that 
the orbifold $Y$ has $16$ Du Val singularities of order $2$, whose resolution gives $16$ disjoint 
$(-2)$-curves in the $K3$ surface $\tilde{Y}$. Let $\tau$ be the non-symplectic automorphism of order $4$. Then the action of $\tau$ on $\tilde{Y}$ is as follows: $\tau^2$ fixes $4$ of the $16$ disjoint
$(-2)$-curves in $\tilde{Y}$, and furthermore, $\tau$ switches $2$ of the $4$ curves of isotropy of order $2$, and leaves each of the remaining $2$ curves invariant. In particular, note that $\tau$ has $4$ fixed points, which are contained in the $2$ invariant $(-2)$-curves. 
\end{example}

\section{Constraints of the singular set} 

\subsection{Group actions on Calabi-Yau homology $K3$ surfaces}
Theorem 1.1 has interesting applications on symplectic finite group actions on symplectic 
$4$-manifolds with torsion canonical class. For an illustration, we shall consider the case of 
symplectic Calabi-Yau homology $K3$ surfaces; the result will be used later in the section.

A well-known property of holomorphic actions on a $K3$ surface is that the fixed-point set does not
contain points of mixed types, i.e., of both Du Val and non-Du Val types. The reason is that the canonical line bundle of a $K3$ surface is holomorphically trivial, meaning that there is a nowhere vanishing holomorphic section. If the induced action on the holomorphic section
is trivial, then all fixed points are Du Val, and if the induced action is nontrivial, none of the fixed
points are Du Val. 

In the following theorem, we generalize this phenomenon to the symplectic category.
For simplicity, we assume the group action is of prime order. 

\begin{theorem}
Let $M$ be a symplectic Calabi-Yau $4$-manifold with $b_1=0$, which is equipped with
a symplectic $G$-action of prime order $p$. Let $X=M/G$ be the quotient orbifold such that 
the resolution $\tilde{X}$ is rational. Then $M$ with the symplectic $G$-action is equivariantly symplectomorphic to the Calabi-Yau cover $Y$ equipped with the symplectic $\Z_p$-action 
of deck transformations {\em(}note that $n=p$ in this case{\em)}. As a consequence, the canonical line bundle $K_M$ admits a nowhere vanishing section $s$, such that the induced action of $G$ on $K_M$ is given by multiplication of $\exp(2\pi i/p)$ for some generator $g\in G$. In particular, the fixed-point set $M^G$ does not contain any fixed points of Du Val type. 
\end{theorem}

\begin{proof}
First of all, since $G$ is of prime order $p$, the singular set of $X=M/G$ consists of $2$-dimensional components $\{\Sigma_i\}$ and isolated points $\{q_j\}$, where $m_i=p$ for each $i$ and $G_j=G$ 
for each $j$. If for some $q_j$, $m_j=1$, then $H_j=G$, and if $m_j>1$, then $H_j$ is trivial and $m_j=p$. Since $\tilde{X}$ is rational, it follows easily that $n=p$, where $n:=\text{lcm}\{m_i,m_j\}$. 

We claim that the set $\{q_j|m_j=1\}$ is empty, and $Y=\tilde{Y}$, which is a symplectic Calabi-Yau $4$-manifold with $b_1=0$. To see this, we first note that the singularities of $Y$ are given by the 
pre-image $\pi^{-1}(q_j)$ where $q_j$ is a singular point of $X$ with $m_j=1$, so the canonical symplectic $\Z_p$-action on $Y$ acts freely on the singular set of $Y$. With this understood, let $x$ be the number of singular points $q_j$ such that $m_j=1$. Then 
$$
\chi(\tilde{Y}/\Z_p)=\chi(M/G)+x(p-1) \mbox{ and } \chi(\tilde{Y}^{\Z_p})=\chi(M^G)-x.
$$
On the other hand, the Lefschetz fixed point theorem implies that
$$
p\cdot \chi(M/G)=\chi(M)+(p-1)\cdot \chi(M^G), \;\;\;\;  p\cdot \chi(\tilde{Y}/\Z_p)=\chi(\tilde{Y})+(p-1)\cdot \chi(\tilde{Y}^{\Z_p}). 
$$
It follows easily that $\chi(\tilde{Y})=\chi(M)+x(p^2-1)$. With $\chi(M)=24$, it follows immediately that 
$\chi(\tilde{Y})=24$ and $x=0$. Hence our claim.

It remains to show that $M$ is $G$-equivariantly symplectomorphic to $Y$ with the natural $\Z_p=G$ action. This part relies on a well-known property of $M$ that $\pi_1(M)$ has no subgroups of finite index (cf. \cite{Bauer, Li}). To finish the proof, we let $pr: M\rightarrow X=M/G$ be the quotient map. We claim that $pr$ can be lifted to a map $\psi: M\rightarrow Y$ under the orbifold covering $\pi: Y\rightarrow X$
from Theorem 1.1. To this end, we need to examine the image of 
$pr_\ast: \pi_1(M)\rightarrow \pi_1^{orb}(X)$, and show that 
$pr_\ast(\pi_1(M))\subset \pi_\ast (\pi_1(Y))$. For this we observe that there is a surjective homomorphism $\rho: \pi_1^{orb}(X)\rightarrow \Z_p$ associated to the orbifold covering 
$\pi: Y\rightarrow X$ such that $\pi_\ast (\pi_1(Y))$ is identified with the kernel of $\rho$. 
With this understood, suppose to the contrary that $pr_\ast(\pi_1(M))$ is not contained in 
$\pi_\ast (\pi_1(Y))$. Then the homomorphism $\rho\circ pr_\ast: \pi_1(M)\rightarrow \Z_p$ must be surjective as $p$ is prime. The kernel of $\rho\circ pr_\ast$ is a subgroup of $\pi_1(M)$
of a finite index, which is a contradiction. Hence our claim that $pr$ can be lifted to a map 
$\psi: M\rightarrow Y$ under the orbifold covering $\pi: Y\rightarrow X$. The map $\psi: M\rightarrow Y$ is clearly an equivariant diffeomorphism, inducing the identity map on the orbifold $X$.
Since the symplectic structure on $Y$ is the pull-back of the symplectic structure on $X$ via the orbifold covering $\pi: Y\rightarrow X$, it follows that $\psi$ is a symplectomorphism. This completes the proof of Theorem 3.1.

\end{proof}

Now we state a theorem which gives some general constraints on the singular set of $X$
(here $X$ is not necessarily a global quotient $M/G$). 

\begin{theorem}
Suppose the symplectic Calabi-Yau $4$-manifold $\tilde{Y}$ is an integral homology $K3$ surface. 
Then the number $n:=\text{lcm} \{m_i,m_j\}$ and the $2$-dimensional components $\{\Sigma_i\}$ of the singular set of $X$ obey the following constraints.
\begin{itemize}
\item [{(1)}] If $p$ is a prime factor of $n$, then $p\leq 19$.
\item [{(2)}] There can be at most one component in $\{\Sigma_i\}$ which has genus greater than
$1$. If there is such a component in $\{\Sigma_i\}$, then the remaining components must be all spheres. Moreover, $n$ must equal the order of isotropy along the component of genus 
$>1$, and if $p$ is a prime factor of $n$, then $p\leq 5$.
\item [{(3)}] There can be at most two components in $\{\Sigma_i\}$ which are torus, and if this 
happens, there are no other components in $\{\Sigma_i\}$, and $n=2$ must be true. 
If there is only one torus in $\{\Sigma_i\}$, then $n$ must equal the order of isotropy along the torus, and moreover, if $p$ is a prime factor of $n$, then $p\leq 11$.
\end{itemize}
\end{theorem}

\begin{proof}
We will prove the theorem by examining the prime order subgroup actions of the symplectic $\Z_n$-action on $\tilde{Y}$. To this end, we let $M$
be a symplectic Calabi-Yau $4$-manifold with $b_1=0$, equipped with a symplectic $G$-action
of prime order $p$. Note that $M$ has the integral homology of $K3$ surface. 

The induced action of $G$ on $H^2(M)$, as an integral $\Z_p$-representation, splits into a direct 
sum of $3$ types of $\Z_p$-representations, i.e, the regular type of rank $p$, the trivial representation of rank $1$, and the representation of cyclotomic type of rank $p-1$. If we let $r,t,s$ be the number
of summands of the above $3$ types of $\Z_p$-representations in $H^2(M)$, then we have the
following identities (cf. \cite{Ed1}):
$$
b_2(M)=rp+t+s(p-1), \chi(M^G)=t-s+2, \mbox{ and } s=b_1(M^G). 
$$
Note that the second identity is the Lefschetz fixed point theorem. As for the third one, i.e., $s=b_1(M^G)$, it was proved in \cite{Ed1} under the assumption
that $M$ is simply connected. However, since its argument is purely cohomological, the identity continues to hold under the weaker condition 
$H_1(M)=0$ (cf. \cite{Ed2}).  As an immediate corollary, note that
if $p$ is a prime factor of $n$, then there is an induced $\Z_p$-action on $\tilde{Y}$. If $p>19$,
the $\Z_p$-representation on $H^2(\tilde{Y})$ can not have any summands of regular type or 
cyclotomic type, i.e., $r=s=0$, because $b_2(\tilde{Y})=22$. In other words, the symplectic $\Z_p$-action on $\tilde{Y}$ is homologically trivial. However, since $c_1(K_{\tilde{Y}})=0$, this is not
possible (cf. \cite{CK1}). Hence part (1) of Theorem 3.2 follows. 

Next we prove part (2) of Theorem 3.2. Let $\Sigma_i$ be a singular component of $X$ whose
genus is denoted by $g_i$
and let $B_i$ be the descendant in the resolution $\tilde{X}$. Applying the adjunction
formula to $B_i$ (note that $c_1(K_{\tilde{X}})=-\sum_i \frac{m_i-1}{m_i} B_i +\sum_j \sum_{k\in I_j} a_{j,k} F_{j,k}$), it follows easily that $B_i^2=2m_i(g_i-1)$. As a consequence, $g_i>1$ if and
only of $B_i^2>0$. Since 
$b_2^{+}(\tilde{X})=1$, it follows immediately that one can have at most one $\Sigma_i$ with 
$g_i>1$. Moreover, suppose there is another component $\Sigma_k$ which is a torus, then its descendent $B_k$ has $B_k^2=0$. It is easy to 
see that $\Sigma_i,\Sigma_k$ can not both exist, because $(B_i+B_k)^2>0$ and $B_i$ and 
$B_i+B_k$ are linearly independent. Hence if there is a singular component $\Sigma_i$ of genus 
$g_i>1$, then all other singular components must be spheres. To see that $n=m_i$ in this case, 
we observe that the pre-image $\pi^{-1}(\Sigma_i)$ in $Y$ has $n/m_i$ many components, each 
is fixed by a subgroup of $\Z_n$ of order $m_i$. The above argument on $X$, if applied to the
orbifold $\tilde{Y}/\Z_{m_i}$, implies immediately that $n/m_i=1$ must be true. Finally, if $p$ is
a prime factor of $n$, then there is a $\Z_p$-action on $\tilde{Y}$ fixing $\pi^{-1}(\Sigma_i)$.
Now observe that in the identity $b_2(\tilde{Y})=rp+t+s(p-1)$, 
$s\geq b_1(\pi^{-1}(\Sigma_i))=2g_i\geq 4$, which implies that $p\leq 22/4+1<7$. 
Hence part (2) of Theorem 3.2 is proved. 

Finally,  we consider part (3) of Theorem 3.2. 

\begin{lemma}
Let $G$ be a finite cyclic group of order $m$, and let $M$ be a symplectic Calabi-Yau $4$-manifold with $b_1=0$. Suppose a symplectic $G$-action on $M$ has at least two fixed components of torus.
Then $m=2$ and $M^G$ consists of the two tori. 
\end{lemma}

\begin{proof}
Let $X=M/G$ be the quotient orbifold. We first consider the special case where $m=p$ is prime.
To this end, we first note that in $b_2(M)=rp+t+s(p-1)$, $s\geq 4$, so that $p\leq 5$. 
To further analyze the $\Z_p$-action for these cases, we shall consider the resolution 
$\tilde{X}$ of $X$, which is a symplectic rational $4$-manifold. 

Let $\{\Sigma_i\}$ be the $2$-dimensional fixed components and $\{q_j\}$ the isolated fixed points 
of the $\Z_p$-action. Let $B_i$ be the descendent of $\Sigma_i$ in $\tilde{X}$, and let 
$D_j\subset \tilde{X}$ be the exceptional set of the minimal resolution of $q_j$. 
With this understood, we shall prove the lemma by looking at the expressions of the
$B_i$'s and the components in the $D_j$'s with respect to a certain basis of $H^2(\tilde{X})$,
which is called a reduced basis. (See Section 4 for more discussions about reduced bases.)

Let $H,E_1,\cdots,E_N$ be a reduced basis. Then 
$c_1(K_{\tilde{X}})=-3H+E_1+\cdots+E_N$ (cf. Section 4). On the other hand, note that 
$$
c_1(K_{\tilde{X}})=-\frac{p-1}{p}\sum_i B_i +\sum_j c_1(D_j).
$$
Denoting by $B_1,B_2$ the two torus components in $\{B_i\}$, 
we write each of $B_1,B_2$ in the reduced basis, with an expression of the form
$aH-\sum_{k=1}^N b_k E_k$, and we shall call the coefficient $a$ in such expressions the
$a$-coefficient of $B_1, B_2$. With this understood, by Lemma 4.2 in \cite{C}, the 
$a$-coefficients of both $B_1,B_2$ are at least $3$. Moreover, if the 
$a$-coefficient equals $3$, then $B_1$ or $B_2$ must take the form 
$B=3H-E_{j_1}-E_{j_2}-\cdots-E_{j_9}$. On the other hand, if there is a symplectic sphere $S$ in 
$\{B_i\}$ or $\{D_j\}$ whose $a$-coefficient is negative, then $S$ must have the homological expression 
$$
S=aH+(|a|+1)E_1-E_{i_1}-\cdots-E_{i_l}
$$ 
for some $a<0$ (cf. \cite{C}, Lemma 3.4).
With this understood, note that 
$$
B\cdot S\leq (|a|+1)+3a=2a+1<0,
$$
which contradicts the fact that $B_1,B_2$ are disjoint from $S$. Hence if there is a symplectic sphere $S$ in $\{B_i\}$ or $\{D_j\}$ whose $a$-coefficient is negative (such a component must be unique, see \cite{C}, Lemma 4.2), the $a$-coefficients of both $B_1,B_2$ must be at least $4$. 

To derive a contradiction for the case where $p=3$ or $5$, we first observe that the contribution of $B_1,B_2$ to the $a$-coefficient of $-p\cdot c_1(K_{\tilde{X}})$ is at least $6(p-1)$, which is greater than $3p$ for $p=3$ or $5$. Hence there must be a sphere $S$ in $\{B_i\}$ or $\{D_j\}$ whose $a$-coefficient is negative. With this understood, the contribution of $B_1,B_2$ to the $a$-coefficient of 
$-p\cdot c_1(K_{\tilde{X}})$ is then at least $8(p-1)$. We will get a contradiction again if the contribution of $S$ to the $a$-coefficient of $-p\cdot c_1(K_{\tilde{X}})$ is greater than $8-5p$. 

Consider first the case of $p=3$. In this case, if $S$ is a component of $\{B_i\}$, then $S$ is a 
$(-6)$-sphere. The $a$-coefficient of $S$ is no less than $-2$ (cf. \cite{C}, Lemma 3.4), and the contribution to $-p\cdot c_1(K_{\tilde{X}})$ is at least $-2(p-1)=-4>8-5p$. If $S$ is a component from 
$\{D_j\}$, then $S$ is a $(-3)$-sphere, and its contribution to the $a$-coefficient of $-p\cdot c_1(K_{\tilde{X}})$ equals $3\times \frac{1}{3}\times (-1)=-1$. In either case, we arrive at
a contradiction. Hence $p=3$ is ruled out. For $p=5$, the argument is similar. If $S$ is a component of $\{B_i\}$, then $S$ is a $(-10)$-sphere. In this case, the contribution of $S$ to the $a$-coefficient of 
$-p\cdot c_1(K_{\tilde{X}})$ is at least $-4(p-1)=-16>8-5p$. If $S$ is a component from $\{D_j\}$, 
there are several possibilities. Note that $D_j$ either consists of a single $(-5)$-sphere, or a pair of 
$(-3)$-sphere and $(-2)$-sphere intersecting transversely at one point. With this understood, note that
$S$ cannot be a $(-2)$-sphere as it has negative $a$-coefficient (cf. \cite{C}, Lemma 3.4). If $S$ is
a $(-5)$-sphere, the contribution of $S$ to the $a$-coefficient of $-p\cdot c_1(K_{\tilde{X}})$ is at least $-6$, and if $S$ is a $(-3)$-sphere, the contribution equals $-2$. In either case, we arrive at
a contradiction. Hence $p=5$ is also ruled out. 

It remains to consider the case of $p=2$. Note that by Theorem 3.1, there are no isolated fixed points, so $\{D_j\}=\emptyset$. We first assume $S$ exists. Then the contribution of $B_1,B_2$ to the 
$a$-coefficient of $-p\cdot c_1(K_{\tilde{X}})$ is at least $8(p-1)=8$. On the other hand, $S$ as a 
component in $\{B_i\}$ must be a $(-4)$-sphere. Its contribution to the $a$-coefficient of 
$-p\cdot c_1(K_{\tilde{X}})$ equals $-1$. This is a contradiction as the $a$-coefficient of 
$-p\cdot c_1(K_{\tilde{X}})$ equals $6$ for $p=2$. Hence $S$ cannot exist, and both $B_1,B_2$ have $a$-coefficient equal to $3$. Then it follows easily that $B_1=B_2=3H-E_{j_1}-E_{j_2}-\cdots-E_{j_9}$ for some classes $E_{j_s}$, $s=1,2,\cdots,9$. On the other hand, $c_1(K_{\tilde{X}})=-3H+E_1+\cdots+E_N$. By comparing with the equation $c_1(K_{\tilde{X}})=-\frac{1}{2}\sum_i B_i$, it follows easily that there are no other components in $\{B_i\}$ besides $B_1,B_2$ (and we must have $N=9$). This proves the lemma for the special case of prime order actions.  

For the general case, it follows easily that $m=2^k$ for some $k>0$. With this understood, observe that if a point $q\in M$ is fixed by some nontrivial element of $G$, then it must be fixed by the subgroup of $G$ of order $2$. It follows easily that the singular set of $X=M/G$ consists of only the two tori. If we continue to denote by $B_1,B_2$ the descendants of the fixed tori in the resolution 
$\tilde{X}$ of $X=M/G$, then we have
$$
c_1(K_{\tilde{X}})=-\frac{m-1}{m} (B_1+B_2).
$$
Again, the $a$-coefficients of $B_1,B_2$ are at least $3$ (cf. \cite{C}, Lemma 4.2), from which the above equation implies that $3\geq \frac{m-1}{m}(3+3)$ by comparing the $a$-coefficients of both sides. It follows immediately that $m=2$, and the proof of the lemma is complete.

\end{proof}

Back to the proof of Theorem 3.2, suppose $\Sigma_1,\Sigma_2$ are two singular components of $X$ which are torus, with $m_1,m_2$ being the order of the isotropy groups respectively. If 
$m_1\neq m_2$, then one of $n/m_1,n/m_2$ must be greater than $1$. Without loss of generality,
assume $n/m_1>1$. Then there are at least two components in $\pi^{-1}(\Sigma_1)\subset \tilde{Y}$, which are fixed by a subgroup of $\Z_n$ of order $m_1$. By Lemma 3.3, we must have $m_1=2$ and $n/m_1=2$. It follows that we must have $n=m_2=4$ by the assumption that $m_1\neq m_2$. 
But this implies that the $\Z_n$-action fixes $\pi^{-1}(\Sigma_2)\subset \tilde{Y}$, so that the subgroup of order $m_1=2$, which already fixes two tori in $\pi^{-1}(\Sigma_1)$, also fixes $\pi^{-1}(\Sigma_2)$. This is clearly a contradiction to Lemma 3.3. Hence $m_1=m_2$. Then the above argument shows that we must have $n/m_1=n/m_2=1$, and $n=2$ by Lemma 3.3. Moreover, there are no other components in $\{\Sigma_i\}$ besides $\Sigma_1,\Sigma_2$.

Finally, suppose there is only one component $\Sigma_1$ which is a torus, with $m_1$ being the order of the isotropy group along $\Sigma_1$. Then if $n>m_1$, there will be at least two components in $\pi^{-1}(\Sigma_1)$, which is fixed by a $\Z_{m_1}$-action on $\tilde{Y}$. By Lemma 3.3, $m_1=2$ and $n/m_1=2$, so that $n=4$. If there is a component in $\{\Sigma_i\}$ with $m_i=n=4$, then this component is also fixed by the $\Z_{m_1}$-action, which contradicts Lemma 3.3. 
Hence there must be a singular point $q_j$ such that $m_j=n=4$. Suppose first that the subgroup $H_j$ at $q_j$ is trivial. Then $\pi^{-1}(q_j)$, consists of one point, is a smooth point in $Y$, and is being fixed by the $\Z_n$-action on $\tilde{Y}$. In particular, it is a fixed point of the subgroup of order $m_1=2$. But this contradicts Lemma 3.3. Suppose $H_j$ is nontrivial. Then $\pi^{-1}(q_j)$ is a singular point of $Y$. Let $D_j$ be the exceptional set of its minimal resolution in $\tilde{Y}$. Then $D_j$ is invariant under the $\Z_n$-action on $\tilde{Y}$. It is easy to see that the action of the subgroup of order $m_1=2$ has a fixed point contained in $D_j$, which is a contradiction to 
Lemma 3.3. This proves that $n$ must be equal to the order of the isotropy group along the unique torus component $\Sigma_1$. Finally, suppose $p$ is a prime factor of $n$. Then the action of the subgroup of $\Z_n$ of order $p$ on $\tilde{Y}$ fixes the torus $\pi^{-1}(\Sigma_1)$. Now appealing to
the identity $b_2(\tilde{Y})=rp+t+s(p-1)$, we find that $s(p-1)\leq 22$, where $s\geq b_1(\pi^{-1}(\Sigma_1))=2$. It follows easily that $p\leq 11$. This completes the proof of Theorem 3.2.

\end{proof}

\subsection{Constraints from Seiberg-Witten-Taubes theory}
In this subsection, we derive some constraints on the singular set of the orbifold $X$ using the Seiberg-Witten-Taubes theory, by extending an argument of T.-J. Li in \cite{Li} to the orbifold setting. 
It is important to note that $X$ is assumed to have $b_1=0$.

The constraints are given in terms of certain numerical contributions of the singular set to the dimension of the moduli space of Seiberg-Witten equations. To describe them, we consider 
any orbifold complex line bundle $L$ over $X$ such that $c_1(L)=0\in H^2(|X|;\Q)$. 
For each singular point $q_j$ of $X$, we denote by $\rho^L_j: G_j\rightarrow \C^\ast$ the 
complex representation of the isotropy group $G_j$ on the fiber of $L$ at $q_j$, and denote by
$\rho_{j,k}(g)$, for $k=1,2$, the eigenvalues of $g\in G_j$ associated to the complex representation of 
$G_j$ on the tangent space $T_{q_j} X$. For each $2$-dimensional singular component $\Sigma_i$,
we denote by $G_i:=\Z_{m_i}$ the isotropy group along $\Sigma_i$, and let $\rho^L_i: G_i
\rightarrow \C^\ast$ be the complex representation of $G_i$ on the fibers of $L$ along $\Sigma_i$,
and let $\rho_i: G_i\rightarrow \C^\ast$ be the complex representation of $G_i$ on the normal bundle
$\nu_{\Sigma_i}$ of $\Sigma_i$. With this understood, we set
$$
I_i(L) := \frac{1}{m_i} \sum_{g\in G_i\setminus \{e\}} \frac{(1+\rho_i(g^{-1}))(\rho^L_i(g)-1)}
{(1-\rho_i(g^{-1}))^2},
$$
and 
$$
I_j(L)
:=\frac{1}{|G_j|}\sum_{g\in G_j\setminus\{e\}} \frac{2(\rho^L_j(g)-1)}{(1-\rho_{j,1}(g^{-1}))
(1-\rho_{j,2}(g^{-1}))}.
$$
It is easy to check that $I_i(L)=I_i(K_X\otimes L^{-1})$, $I_j(L)=I_j(K_X\otimes L^{-1})$ for any $i,j$. 
Finally, we set 
$$
d(L):=\sum_i I_i(L)\chi(\Sigma_i)+\sum_j I_j(L).
$$
One can easily check that, with $c_1(L)=0$, and with 
$$
c_1(\nu_{\Sigma_i}) [\Sigma_i]=\Sigma_i^2=2g_i-2=-c_1(T\Sigma_i) [\Sigma_i]=-\chi(\Sigma_i)
$$ 
by the adjunction formula (here $g_i$ is the genus of $\Sigma_i$), 
$d(L)$ equals the dimension of the moduli space of Seiberg-Witten equations associated 
to the orbifold complex line bundle $L$ (cf. \cite{C0}, Appendix A). With this understood, we note  
that $d(L)=d(K_X\otimes L^{-1})$, and $d(L)=0$ if $L$ is the trivial complex line bundle or
$L=K_X$. 

\begin{theorem}
Suppose $b_1(X)=0$. Then for any orbifold complex line bundle $L$ such that 
$c_1(L)=0\in H^2(|X|;\Q)$, one has 
$$
d(L)\leq 0,
$$
with ``$=$" if and only if $L$  is the trivial complex line bundle or $L=K_X$. 
\end{theorem}

\begin{proof}
We begin by noting that $b_2^{+}(X)=b_2^{+}(\tilde{X})=1$. With $b_1(X)=0$, the wall-crossing number for the Seiberg-Witten invariant of
$X$ equals $\pm 1$. With this understood, we denote by $SW_X(L)$ the Seiberg-Witten invariant 
of $X$ associated to an orbifold complex line bundle $L$ defined in the Taubes chamber. Then if
the dimension $d(L)\geq 0$, one has 
$$
|SW_X(L)-SW_X(K_X\otimes L^{-1})|=1.
$$
Now observe that if $SW_X(L)\neq 0$ and $c_1(L)=0$, $L$ must be the trivial orbifold complex line
bundle (cf. \cite{T}). Since $K_X$ is torsion of order $n>1$ (recall that $n=\text{lcm} \{m_i,m_j\}$ 
is the minimal
number such that $K_X^n$ is trivial, cf. Lemma 2.1), it is clear that one of the orbifold complex line bundles, $L$ or $K_X\otimes L^{-1}$, must be a non-trivial torsion bundle, hence has vanishing Seiberg-Witten invariant. It follows easily that one of $SW_X(L)$, $SW_X(K_X\otimes L^{-1})$ must equal $\pm 1$. This implies that either $L$ or $K_X\otimes L^{-1}$ must be the trivial orbifold complex line bundle. Theorem 3.4 follows easily.  

\end{proof}
 
Suppose $n=\text{lcm} \{m_i,m_j\}>2$. Then $L=K_X^k$ is nontrivial and not equal to $K_X$ for $k=2,3,\cdots,n-1$. By Theorem 3.4, $d(K_X^k)<0$ for any $2\leq k\leq n-1$. On the other hand, 
note that a singular component $\Sigma_i$ makes zero contribution to $d(L)$ for any $L$ if 
$\Sigma_i$ is a torus, and one can check directly that $I_j(K_X^k)=0$ for any $k$ if $q_j$ is a Du Val singularity (i.e., $m_j=1$). It is easy to see that we have the following
 
 \begin{corollary}
 Suppose $n=\text{lcm} \{m_i,m_j\}>2$. Then the following are true. 
 \begin{itemize}
 \item [{(1)}] For each $2\leq k\leq n-1$, $d(K_X^k)$ is a negative, even integer.
 \item [{(2)}] Either there is a singular component $\Sigma_i$ which is not a torus, or there is a singular point $q_j$ which is non-Du Val. 
 \end{itemize}
 \end{corollary}
 
 \noindent{\bf Remarks:}
 It is possible that for a singular point $q_j$, the number $I_j(K_X^k)>0$ for some $2\leq k\leq n-1$;
 this depends on the isotropy type of $q_j$.
 For example, if $q_j$ is of isotropy of order $5$ of type $(1,1)$, then
 $$
 I_j(K_X^2)=\frac{1}{5}\sum_{1\neq \lambda\in\C^\ast, \lambda^5=1}\frac{2(\lambda^4-1)}{(1-\lambda)(1-\lambda)}=\frac{2}{5}.
 $$
So the conditions $d(K_X^k)<0$ give rise to nontrivial constraints on the singular set.

\section{A successive blowing-down procedure}

In this section, we describe a general successive symplectic blowing-down procedure. 
First, we shall adopt the following notations: we set $X_N:=\C\P^2\# N \overline{\C\P^2}$, 
which is equipped with a symplectic structure denoted by $\omega_N$. In order to emphasize the dependence of the canonical class on the symplectic structure, we shall denote by $K_{\omega_N}$ the canonical line bundle of $(X_N,\omega_N)$.

\subsection{The statement}
The successive symplectic blowing-down procedure, to be applied to $(X_N,\omega_N)$ for 
$N\geq 2$, depends on a choice of the so-called {\bf reduced basis} of $(X_N,\omega_N)$. 
To explain this notion, we let $\E_{X_N}$ be the set of classes in $H^2(X_N)$ which can be represented by a smooth 
$(-1)$-sphere, and let $\E_{\omega_N}:=\{E\in \E_{X_N}| c_1(K_{\omega_N})\cdot E= -1\}$. 
Then each class in $\E_{\omega_N}$ can be represented by a symplectic $(-1)$-sphere; 
in particular, $\omega_N(E)>0$ for any $E\in\E_{\omega_N}$. With this understood, a basis 
$H, E_1,E_2,\cdots, E_N$ of $H^2(X_N)$ is called a reduced basis of $(X_N,\omega_N)$ if the following are true:
\begin{itemize}
\item it has a standard intersection form, i.e., $H^2=1$, $E^2_i=-1$ and $H\cdot E_i=0$ for any $i$, and $E_i\cdot E_j=0$ for any $i\neq j$; 
\item $E_i\in\E_{\omega_N}$ for each $i$, and moreover, the following area conditions are satisfied
for $N\geq 3$: $\omega_N(E_N)=\min_{E\in\E_{\omega_N}}\omega_N(E)$, and for any $2<i<N$, 
$\omega_N(E_i)=\min_{E\in\E_i}\omega_N(E)$, where
$\E_i:=\{E\in\E_{\omega_N}| E\cdot E_j=0 \;\; \forall j>i\}$ for any $i<N$; 
\item $c_1(K_{\omega_N})=-3H+E_1\cdots+E_N$.
\end{itemize}

Reduced bases always exist. Moreover, if we assume $\omega(E_1)\geq \omega(E_2)$
without loss of generality, then a reduced basis $H, E_1,E_2,\cdots, E_N$ obeys the following constraints in symplectic area (cf. \cite{LW}): 
\begin{itemize}
\item $\omega_N(H)>0$, and for any $j>i$, $\omega_N(E_i)\geq \omega_N(E_j)$;
\item for any $i\neq j$, $H-E_i-E_j\in\E_{\omega_N}$, so that $\omega_N(H-E_i-E_j)>0$; 
\item $\omega_N(H-E_i-E_j-E_k)\geq 0$ for any distinct $i,j,k$.
\end{itemize}
We remark that a reduced basis is not necessarily unique, however, the symplectic areas of its 
classes 
$$(\omega_N(H),\omega_N(E_1), \omega_N(E_2), \cdots, \omega_N(E_N))$$ 
uniquely determine the symplectic structure $\omega_N$ up to symplectomorphism, cf. \cite{KK}. 

\begin{definition}
The symplectic structure $\omega_N$ is called {\bf odd} if $\omega_N(H-E_1-2E_2)\geq 0$, and is called {\bf even} if otherwise.
\end{definition}

We remark that since $\omega_N$ is determined by 
$\omega_N(H),\omega_N(E_1), \omega_N(E_2), \cdots, \omega_N(E_N)$
up to symplectomorphism, the above definition does not depend on the choice of the reduced basis.

The following technical result from \cite{KK} is crucial to our construction. 

\vspace{1.5mm}

{\it Suppose $N\geq 2$. Then for {\bf any} $\omega_N$-compatible almost complex structure $J$, any class $E\in\E_{\omega_N}$ which has the minimal symplectic area can be represented by an embedded $J$-holomorphic sphere. In particular, for $N\geq 3$, the class $E_N$ in a reduced basis 
$H, E_1,\cdots, E_N$ can be represented by a $J$-holomorphic $(-1)$-sphere for any given $J$.
}
\vspace{1.5mm}

With the preceding understood, the following lemma makes it possible for a successive blowing-down procedure. 

\begin{lemma}
Let $H, E_1,\cdots,E_N$ be a reduced basis of $(X_N,\omega_N)$, and let $C_N$ be any symplectic
$(-1)$-sphere in $(X_N,\omega_N)$ representing the class $E_N$. Denote by $(X_{N-1},\omega_{N-1})$ the symplectic blowdown of $(X_N,\omega_N)$ along $C_N$. Then $H, E_1,\cdots,E_{N-1}$
naturally descend to a reduced basis $H^\prime, E_1^\prime,\cdots,E_{N-1}^\prime$ of $(X_{N-1},\omega_{N-1})$. When $N\geq 3$, $\omega_{N-1}$ is odd if and only if $\omega_N$ is odd. 
\end{lemma}

\begin{proof}
It is clear that $H, E_1,\cdots,E_{N-1}$ naturally descend to a basis $H^\prime, E_1^\prime,\cdots,E_{N-1}^\prime$ of $H^2(X_{N-1})$. We need to show that it is a reduced basis of $(X_{N-1},\omega_{N-1})$, and moreover, when $N\geq 3$, $\omega_{N-1}$ is odd if and only if $\omega_N$ is odd. 

First of all, we note that $H^\prime, E_1^\prime,\cdots,E_{N-1}^\prime$ has the standard intersection form, and the symplectic
canonical class of $(X_{N-1},\omega_{N-1})$ is given by 
$$
c_1(K_{\omega_{N-1}})=-3H^\prime+E_1^\prime+\cdots+E_{N-1}^\prime.
$$
It remains to verify that for each $i$, $E_i^\prime\in \E_{\omega_{N-1}}$, and moreover, the following area conditions 
are satisfied: $\omega_{N-1}(E_{N-1}^\prime)=\min_{E^\prime\in \E_{\omega_{N-1}}}\omega_{N-1}(E^\prime)$, 
and for any $i<N-1$, $\omega_{N-1}(E_{i}^\prime)=\min_{E^\prime\in \E_{i}^\prime}\omega_{N-1}(E^\prime)$, where 
$\E_i^\prime:=\{E^\prime\in\E_{\omega_{N-1}}| E^\prime\cdot E_j^\prime=0, \forall j>i\}$. 

The key step is to show that the set $\E_{N-1}=\{E\in\E_{\omega_N}|E\cdot E_N=0\}$ may be identified with the set
$\E_{\omega_{N-1}}$ by identifying the elements of $\E_{N-1}$ with their descendants in $H^2(X_{N-1})$, and moreover,
under this identification the symplectic forms $\omega_N=\omega_{N-1}$. To see this, let $E\in \E_{N-1}$ be any class 
and let $E^\prime$ be its descendant in $H^2(X_{N-1})$. We choose a $J_1$ such that $C_N$ is $J_1$-holomorphic. 
Then pick a generic $J_0$ and connect $J_0$ and $J_1$ through a smooth path $J_t$. Since $J_0$ is generic, $E$ can
be represented by a $J_0$-holomorphic $(-1)$-sphere, denoted by $C_E$. On the other hand, since $E_N$ has minimal
symplectic area, for each $t$, $E_N$ is represented by a $J_t$-holomorphic $(-1)$-sphere $C_t$, which depends on $t$
smoothly, with $C_1$ at $t=1$ being the original $(-1)$-sphere $C_N$. Note also that the $J_0$-holomorphic $(-1)$-spheres
$C_E$ and $C_0$ are disjoint because $E\cdot E_N=0$. With this understood, we note that the isotopy from $C_0$ to $C_1=
C_N$ is covered by an ambient isotopy $\psi_t: X_N\rightarrow X_N$, where each $\psi_t$ is a symplectomorphism
(cf. Proposition 0.3 in \cite{ST}). It follows easily that $E$ is represented by the symplectic $(-1)$-sphere $\psi(C_E)$, which is
disjoint from $C_N$. This shows that the descendant $E^\prime$, which is represented by the symplectic $(-1)$-sphere 
$\psi(C_E)$ in $X_{N-1}$, lies in the set $\E_{\omega_{N-1}}$. Moreover, $\omega_N(E)=\omega_{N-1}(E^\prime)$.
Finally, let $E^\prime$ be any class in $\E_{\omega_{N-1}}$. Then $E^\prime$ can be represented by a smooth $(-1)$-sphere,
to be denoted by $S^\prime$, and $E^\prime \cdot c_1(K_{\omega_{N-1}})=-1$. Now recall that the $4$-manifold $X_{N-1}$ 
is obtained from $X_N$ by removing the $(-1)$-sphere $C_N$ and then filling in a symplectic $4$-ball $B$. Without loss of
generality, we may assume $S^\prime$ is lying outside $B$, because if otherwise, one can always apply an ambient isotopy
to push $S^\prime$ outside of $B$. With this understood, the smooth sphere $S^\prime$ can be lifted to a smooth sphere $S$ in
$X_N$. Let $E$ be the class of $S$. Then clearly $E\cdot E_N=0$ and $E^\prime$ is the descendant of $E$ in $H^2(X_{N-1})$.
To see that $E\in \E_{N-1}$, we only need to verify that $E\cdot c_1(K_{\omega_N})=-1$. But this follows easily from the fact that
$c_1(K_{\omega_N})=c_1(K_{\omega_{N-1}})+E_N$ and $E^\prime \cdot c_1(K_{\omega_{N-1}})=-1$. Hence the claim that
$\E_{N-1}$ and $\E_{\omega_{N-1}}$ are naturally identified and the symplectic forms $\omega_N$ and $\omega_{N-1}$ agree. 

With the preceding understood, it follows easily that for each $i=1,2,\cdots,N-1$, $E_i^\prime\in \E_{\omega_{N-1}}$. Moreover,
$\omega_{N-1}(E_{N-1}^\prime)=\min_{E^\prime\in \E_{\omega_{N-1}}}\omega_{N-1}(E^\prime)$. We further observe that for
each $i<N-1$, the subset $\E_i$ of $\E_{N-1}$ is identified with the subset $\E_i^\prime$ of $\E_{\omega_{N-1}}$ under the 
identification between $\E_{N-1}$ and $\E_{\omega_{N-1}}$. 
With $\omega_N$ and $\omega_{N-1}$ agreeing with each other under
the identification, it follows immediately that $H^\prime, E_1^\prime,\cdots,E_{N-1}^\prime$
is a reduced basis of $(X_{N-1},\omega_{N-1})$. Moreover, when $N\geq 3$, $\omega_{N-1}$ is odd if and only if $\omega_N$ is odd. 
This finishes off the proof.

\end{proof}

For simplicity, we shall continue to use the notations $H,E_1,\cdots,E_{N-1}$ to denote the descendants in the symplectic blowdown $(X_{N-1},\omega_{N-1})$, instead of the notations 
$H^\prime, E_1^\prime,\cdots,E_{N-1}^\prime$ in the lemma.

Now fixing any reduced basis $H, E_1,E_2,\cdots, E_N$, we can successively blow down the classes $E_N$, $E_{N-1}, \cdots, E_3$, reducing $(X_N,\omega_N)$ to $(X_2,\omega_2)$.
To further blow down $(X_2,\omega_2)$, we note that $\omega_N$ is odd if and only if $E_2$ has the minimal area among the classes in $\E_2$, i.e., $\omega_N(E_2)=\min_{E\in\E_2}\omega_N(E)$,
as it is easy to see that
$$
\E_2=\{E_1,E_2, H-E_1-E_2\}.
$$
Thus when $\omega_N$ is odd, we can further blow down $E_2$ to reach 
$\C\P^2\# \overline{\C\P^2}$. If $\omega_N$ is even, then 
$\omega_N(H-E_1-E_2)=\min_{E\in\E_2}\omega_N(E)$. In this case, by blowing down the 
$(-1)$-class $H-E_1-E_2$, we reach the final stage $\s^2\times \s^2$. 

Since at each stage of the blowing-down procedure the $(-1)$-class has minimal area, it can be 
represented by a $J$-holomorphic $(-1)$-sphere for any given $J$. This property allows us to 
construct, in a canonical way, the descendant of a given set of symplectic surfaces $D=\cup_k F_k$ in $(X_N,\omega_N)$ under the successive blowing-down procedure. Without loss of much generality, we shall assume $D$ satisfies the following condition:
\begin{itemize}
\item [{(\dag)}] Any two symplectic surfaces $F_k,F_l$ in $D$ are either disjoint, or intersect transversely and positively at one point, and no three distinct components of $D$ meet in one point. 
\end{itemize}

Further assumptions on $D$ are required so that the procedure is reversible. 
In order to explain this, observe that the class of each $F_k$ in $D$ can be written 
with respect to the reduced basis $H, E_1, E_2,\cdots,E_N$ in the following form: 
$$
F_k=aH-\sum_{i=1}^N b_i E_i, \mbox{ where $a,b_i\in\Z$}.
$$ 
We shall call the numbers $a$ and $b_i$ the {\it $a$-coefficient} and {\it $b_i$-coefficients} of $F_k$. (See Section 3 of \cite{C} for some general properties of the $a$-coefficient and $b_i$-coefficients.) The expression $F_k=aH-\sum_{i=1}^N b_i E_i$ is called the {\it homological expression} of $F_k$ (with respect to the reduced basis). 

With the preceding understood, the assumptions on $D$ are concerned with the homological expressions of the components $F_k$ whose $a$-coefficients are zero. More concretely, 
it is known (cf. \cite{C}, Lemma 3.3) that such a component must be a symplectic sphere, and its $b_i$-coefficients are equal to $1$ except for one of them, which equals $-1$. We shall call the $E_i$-class with the $(-1)$ $b_i$-coefficient the {\it leading class} of the component. With this understood, it is easy to show that for any given component $S$ of $D$
which has zero $a$-coefficient, there are at most two components $F_k$ in $D$ such that 
the expression of $F_k$ contains the leading class of $S$ and $F_k$ has a zero $a$-coefficient 
(cf. Lemma 4.5). The assumptions we shall impose on $D$ are concerned with the homological expressions of such components $F_k$ for any given such $S$ in $D$. 

To be more precise, let $S\subset D$ be any such symplectic sphere, and we write the homological expression of $S$ as 
$$
S=E_n-E_{l_1}-E_{l_2}-\cdots-E_{l_\alpha}, \mbox{ where  $n<l_s$ for all $s$}. 
$$
Then the imposed assumptions on $D$ are stated as follows:
\begin{itemize}
\item [{(a)}] Suppose there are two symplectic spheres $S_1,S_2\subset D$ whose $a$-coefficients equal zero and whose homological expressions contain the leading class $E_n$ of $S$. Then for any class $E_{l_s}$ which appears in $S$, but appears in neither $S_1$ nor $S_2$, there is at most one component $F_k$ of $D$ other than $S$, whose homological expression contains $E_{l_s}$ with 
$F_k\cdot E_{l_s}=1$.
\item [{(b)}] Suppose there is only one symplectic sphere $S_1\subset D$ whose $a$-coefficient 
equals zero and whose homological expression contains the leading class $E_n$ of $S$. 
Then there is at most one class $E_{l_s}$ in $S$, which does not appear in $S_1$, but either appears in the expressions of more than one components $F_k\neq S$, or appears in the expression of only one component $F_k\neq S$ but with $F_k\cdot E_{l_s}>1$.
\end{itemize}
(We remark that when $S$ is a $(-2)$-sphere or $(-3)$-sphere, and $S_1,S_2$ are disjoint from $S$,
the assumptions (a) and (b) are automatically satisfied.)

With the preceding understood, we now state the theorem concerning the descendant of $D$ under
the successive symplectic blowing-down procedure. For simplicity, we shall only discuss the case where the symplectic structure $\omega_N$ is odd, which is the most relevant case for us. The case
where $\omega_N$ is even can be similarly dealt with. 

\begin{theorem}
Let $D=D_N=\cup_k F_k$ be a union of symplectic surfaces in $(X_N,\omega_N)$, where 
$N\geq 2$ and $\omega_N$ is odd, such that $D_N$ satisfies the condition {\em(}\dag{\em)}. Fix any reduced 
basis $H,E_1,E_2,\cdots,E_N$ of $(X_N,\omega_N)$ such that the assumptions 
$\left(a\right)$ and $\left(b\right)$ are satisfied for the homological expressions of the components $F_k$ in $D_N$. We set 
$$
\E_0(D_N)=\{E_i|\mbox{there is no $F_k\subset D_N$ with zero $a$-coefficient such that $E_i\cdot F_k>0$}\}.
$$
Then 
there is a well-defined successive symplectic blowing-down procedure associated to the reduced basis, blowing down the classes $E_N,E_{N-1},\cdots,E_2$ successively, such that $(X_N,\omega_N)$ is reduced to 
$(X_1,\omega_1)$ {\em(}note that $X_1=\C\P^2\# \overline{\C\P^2}${\em)}, and $D_N$ is transformed to its descendant $D_1$ in $(X_1,\omega_1)$, which is a union of $J_1$-holomorphic curves with respect to some $\omega_1$-compatible almost complex structure $J_1$ on $(X_1,\omega_1)$, where the singularities and the intersection pattern of the components of $D_1$ are canonically determined by the homological expressions of the components $F_k$ of $D_N$. Moreover, under any of the conditions $\left(c\right),\left(d\right),\left(e\right)$ listed below, one can further blow down the class $E_1$
to reach $\C\P^2$ in the final stage of the successive blowing-down, with the descendant $D_0$ of $D_1$ in $\C\P^2$ having the same properties of $D_1$:
\begin{itemize}
\item [{(c)}] The classes $E_1,E_2$ have the same area, i.e., $\omega_N(E_1)=\omega_N(E_2)$.
\item [{(d)}] The class $E_1$ is the leading class of a symplectic sphere $S\subset D_N$. 
\item [{(e)}] There is a component $F_k=aH-bE_1-\sum_{i>1}b_i E_i$ of $D_N$ such that $2b<a$.
\end{itemize}

More specifically, let $\E(D_N):=\E_0(D_N)\setminus \{E_1\}$ if the final stage of the successive blowing-down is $(X_1,\omega_1)$ and let $\E(D_N):=\E_0(D_N)$ if the final stage is $\C\P^2$. Then the new intersection points in $D_1$ or $D_0$ are labelled by the elements of $\E(D_N)$. For each new intersection point $\hat{E}_i$ labelled by $E_i\in\E(D_N)$, there is a small $4$-ball $B(\hat{E}_i)$ centered at $\hat{E}_i$, with standard symplectic structure and complex structure, such that $D_1\cap B(\hat{E}_i)$ or $D_0\cap B(\hat{E}_i)$ consists of a union of holomorphic discs intersecting at 
$\hat{E}_i$, which are either embedded or singular at $\hat{E}_i$ with a singularity modeled by equations of the form  $z_1^n=a z_2^m$ in some compatible complex coordinates $(z_1,z_2)$ {\em(}i.e., the link of the singularity is always a torus knot{\em)}. The orders of tangency of the intersections at $\hat{E}_i$ as well as the singularity types in $B(\hat{E}_i)$ are completely and canonically determined by the pattern of appearance of the class $E_i$ and the classes not contained in 
$\E(D_N)$ in the homological expressions of the components $F_k$ in $D_N$. Finally, a component of $D_N$ descends to a component in $D_1$ or $D_0$ if and only if it has nonzero $a$-coefficient {\em(}a component with zero $a$-coefficient disappears{\em)}. 
\end{theorem}

{\bf Remarks:} (1) We shall call $D_0$ or $D_1$ a {\bf symplectic arrangement} of 
pseudoholomorphic curves. (We borrow the terminology from \cite{RuStar}, where in the case when
$D_0$ is a union of degree $1$ pseudoholomorphic spheres in $\C\P^2$, it is called a symplectic line arrangement.)

(2) Two situations of the new intersection points are worth mentioning, as they occur more generically: let $E_n\in\E(D_N)$ be any element. 
\begin{itemize}
\item [{(i)}] If $E_n$ is not the leading class of any symplectic sphere in $D_N$, 
then the descendants of the components of $D_N$ containing $E_n$ will intersect the $4$-ball $B(\hat{E}_n)$ in a union of holomorphic discs, which are all embedded and intersecting 
at $\hat{E}_n$ transversely.
\item [{(ii)}] If $E_n$ is the leading class of a symplectic sphere $S\subset D_N$, where
$$
S=E_n-E_{l_1}-E_{l_2}-\cdots-E_{l_\alpha}, 
$$
such that the classes $E_{l_s}$ in $S$ are not the leading class of any symplectic spheres in $D_N$,
then the holomorphic discs in $B(\hat{E}_n)$ are all embedded, and moreover, each $E_{l_s}$ determines a complex line (through the origin $\hat{E}_n$) in $B(\hat{E}_n)$, such that the descendants of the components of $D_N$ containing $E_{l_s}$ will intersect the $4$-ball $B(\hat{E}_n)$ in a union of holomorphic discs which are all tangent to the complex line determined by $E_{l_s}$, with a tangency of order $2$. 
\end{itemize}

(3) The successive blowing-down procedure is purely a symplectic operation; there are no 
holomorphic analogs. Note that the descendant $D_0$ or $D_1$ depends on the choice of the reduced basis, which in general is not necessarily uniquely determined by the symplectic structure 
$\omega_N$.
On the other hand, there is also flexibility in choosing the symplectic structure $\omega_N$ 
(cf. \cite{C}, Lemma 4.1). Hence it is not clear if there is a descendant $D_0$ or $D_1$ that is determined by $D_N$ itself. 

(4) The successive blowing-down procedure is reversible; by reversing the procedure (with either 
symplectic blowing-up or holomorphic blowing-up), one can recover $D_N\subset X_N$ up to a smooth isotopy. Note that in each step of the reversing successive blowing-up operation, one either takes 
the total transform or the proper transform (cf. \cite{MW}), depending on whether in the corresponding
blowing-down step, the $(-1)$-sphere being blown down is part of the descendant of $D_N$ or not. 

\subsection{The construction} Suppose we are given with a union of symplectic surfaces 
$D=D_N=\cup_k F_k$ in $(X_N,\omega_N)$ satisfying the condition (\dag). 
We shall first describe how to blow down $(X_N,\omega_N)$ along the class $E_N$ and how to
define the descendants of the components $F_k$ of $D_N$ in $(X_{N-1},\omega_{N-1})$.
First of all, we slightly perturb the symplectic surfaces $F_k$ if necessary, so that the intersection 
of $F_k$ is $\omega_N$-orthogonal (cf. \cite{Gompf}).
Furthermore, we choose an $\omega_N$-compatible almost complex structure $J_N$ which is integrable near each intersection point of the symplectic surfaces $F_k$ such that $D_N$ is $J_N$-holomorphic. With this understood, since $N\geq 2$ and $\omega_N$ is odd, we may represent the class $E_N$ by an embedded $J_N$-holomorphic sphere $C_N$. 

\subsubsection{Perturbing the $(-1)$-spheres to a general position}
An important feature of the successive blowing-down procedure is that, before we blow down the
$(-1)$-sphere $C_N$, we shall first put it in a general position, as long as $C_N$ is not part of $D_N$. We carry out this step as follows. 

The intersection of $C_N$ with each $F_k$ is isolated, though not necessarily transverse, and furthermore, $C_N$ may contain the intersection points of the components $F_k$ in $D_N$.
The local models for the intersection of $C_N$ with $D_N$ are as follows. If $p\in C_N\cap D_N$ is the intersection of $C_N$ with a single component $F_k$, then locally near $p$, $C_N$ and $F_k$ are given respectively by $z_2=0$ and $z_2=z_1^m + \text{higher order terms}$. If $p\in C_N\cap D_N$ is the intersection of $C_N$ with more than one components of $D_N$, then near $p$ there is a standard holomorphic coordinate system such that the relevant components of $D_N$ are given by complex lines through the origin, and $C_N$ is given by an embedded holomorphic disc through the origin. With this understood, it is easy to see that one can always slightly perturb $C_N$ to a symplectic $(-1)$-sphere, still denoted by $C_N$ for simplicity, such that $C_N$ obeys the following
{\bf general position condition:}

\vspace{1.5mm}

{\it $C_N$ intersects each $F_k$ transversely and positively, and $C_N$ does not contain any intersection points of the components of $D_N$. Furthermore, the intersection of $C_N$ with each $F_k$ is $\omega_N$-orthogonal {\em(}after a small perturbation if necessary, cf. \cite{Gompf}{\em)}. 
{\em(}We should point out that when $C_N$ is part of $D_N$, there is no need to perturb 
$C_N$.{\em)}
}

\vspace{1.5mm}

By the Weinstein neighborhood theorem, a neighborhood $U$ of $C_N$ is symplectically modeled by
a standard symplectic structure on a disc bundle associated to the Hopf fibration, where $C_N$ is identified with the zero-section. With this understood, for each $F_k$ which intersects $C_N$, 
we slightly perturb $F_k$ near the intersection points so that $F_k$ coincides with a fiber disc inside $U$. Now symplectically blowing down $(X_N,\omega_N)$ along $C_N$ amounts to cutting $X_N$ open along $C_N$ and then inserting a standard symplectic
$4$-ball of a certain radius back in (the radius of the $4$-ball is determined by the area of $C_N$). We denote the resulting symplectic $4$-manifold by $(X_{N-1},\omega_{N-1})$. Then the descendant of 
$F_k$ in $X_{N-1}$ is defined to be the symplectic surface, to be denoted by $\tilde{F_k}$, which is obtained by adding a complex linear disc to $F_k\setminus C_N$ inside the 
standard symplectic $4$-ball for each of the intersection points of $F_k$ with $C_N$. 
If $F_l$ is another symplectic surface intersecting $C_N$, then the descendant $\tilde{F_l}$ of $F_l$ in $X_{N-1}$ will intersect with 
$\tilde{F_k}$ at the origin of the standard symplectic $4$-ball, which is the only new intersection point introduced to $F_k, F_l$ under the blowing down operation along $C_N$. We denote the origin of the standard symplectic $4$-ball by $\hat{E}_N\in X_{N-1}$. 
Note that under this construction, $\tilde{F_k}$ is immersed in general, where the (transverse) self-intersection at $\hat{E}_N$ is introduced
if $F_k$ intersects $C_N$ at more than one point. Finally, we denote by $B(\hat{E}_N)$ a small $4$-ball centered at $\hat{E}_N$
such that $B(\hat{E}_N)\cap (\cup_k \tilde{F_k})$ consists of a union of (linear) complex discs through the origin. Note that for each $k$, the number of complex discs in $B(\hat{E}_N)\cap \tilde{F_k}$ equals the intersection number $E_N\cdot F_k$.

To continue with the successive blowing-down procedure, we consider the union of the generally immersed symplectic surfaces $D_{N-1}:=\cup_k \tilde{F_k}$ in $(X_{N-1},\omega_{N-1})$. For simplicity, we shall continue to denote the descendant $\tilde{F_k}$ by the original notation $F_k$.  However, one should note that the initial condition (\dag) concerning the intersections of the components $F_k$ of $D_N$ is replaced by the following condition:
\begin{itemize}
\item [{(\ddag)}] There exists an $\omega_{N-1}$-compatible almost complex structure 
$J_{N-1}$ such that each component $F_k$ in $D_{N-1}$ is $J_{N-1}$-holomorphic, 
self-intersecting and intersecting with each other transversely. Moreover, $J_{N-1}$ is integrable near the intersection points. 
\end{itemize}

We shall continue this process if $N-1\geq 2$. Now suppose we are at the stage of 
$(X_n,\omega_n)$ for some $n<N$, with the descendant of $D_N$ in $X_n$ denoted by $D_n$, 
which is $J_n$-holomorphic with respect to some $\omega_n$-compatible almost complex structure
$J_n$. Suppose $n\geq 2$ and we are trying to blow down the class $E_n$ in the reduced basis
of $(X_n,\omega_n)$, and to define the descendant of $D_n$ under the blowing-down operation. 
To this end, we represent the class $E_n$ by a $J_n$-holomorphic sphere $C_n$. If $C_n$ is not
part of $D_n$, then as we argued in the case of $C_N$, one can slightly perturb $C_n$ to a 
symplectic $(-1)$-sphere, still denoted by $C_n$, such that $C_n$ obeys the general position condition. With this understood, we simply blow down $(X_n,\omega_n)$ along $C_n$ in the same way as we blow down $(X_N,\omega_N)$ along $C_N$, and move on to the next stage $(X_{n-1},\omega_{n-1})$. 

However, if $C_n$ is part of $D_n$, then we can no longer perturb $C_n$ before blowing it down, in order to make the successive blowing-down procedure reversible. In the easy situation where $C_n$
is one of the original symplectic surfaces in $D_N$, we can simply blow it down without perturbing it.
In general, $C_n$ is the descendant of a symplectic sphere $S\subset D_N$ to $X_n$, 
where the $a$-coefficient of $S$ is zero and the class $E_n$ appears in $S$ as the leading class, i.e., $S$ has the homological expression 
$$
S=E_n-E_{l_1}-\cdots-E_{l_\alpha}, \mbox{ where $n<l_s$ for all $s$.}
$$
In this case, more care needs to be given in defining the descendant $D_{n-1}$ of $D_n$ in the next stage $(X_{n-1},\omega_{n-1})$. 

\subsubsection{Tangency of higher orders and singularities}
When $C_n$ is part of $D_n$, intersection of higher order tangency as well as singularities 
may occur in $D_{n-1}$. In order to construct $D_{n-1}$, we need the following technical lemma.

\begin{lemma}
Let $(M,\omega)$ be a symplectic $4$-manifold and $C$ be a symplectic $(-1)$-sphere in 
$(M,\omega)$. Let $(M^\prime,\omega^\prime)$ be the symplectic blow-down of $(M,\omega)$ along $C$, obtained by removing 
$C$ and gluing back a standard symplectic $4$-ball {\em(}with an appropriate size depending on the area of $C${\em)}. Note that the set of points on $C$ corresponds naturally to the set of complex lines through 
the origin in the standard symplectic $4$-ball in $(M^\prime,\omega^\prime)$. With this understood, the following statements hold.
\begin{itemize}
\item [{(1)}] Let $S_0,S_1,\cdots,S_k$ be symplectic surfaces in $(M,\omega)$, which intersect $C$ at a point $p$.
Moreover, suppose there is a complex coordinate system $(w_1,w_2)$ centered at $p$ in which the symplectic structure
$\omega$ is standard, such that $C$ is defined by $w_2=0$, $S_0$ is defined by $w_1=0$, and each $S_i$, $i>0$, is 
defined by the complex line $w_2=a_iw_1$ for some distinct complex numbers $a_i\neq 0$. Then the 
descendant $S_i^\prime$
of $S_i$ in the blow-down $(M^\prime,\omega^\prime)$ can be defined as follows: let $(z_1,z_2)$ be the complex coordinates 
of the standard symplectic $4$-ball in $(M^\prime,\omega^\prime)$, such that the complex line corresponding to the intersection 
point $p\in C$ is given by $z_1=0$, then $S_0^\prime$ is obtained by gluing a complex disc to $S_0\setminus C$ contained in
$z_1=0$, and for each $i>0$, $S_i^\prime$ is obtained by gluing a holomorphic disc to $S_i\setminus C$ defined by the
equation $z_1=b_i z_2^2$ for some distinct complex numbers $b_i\neq 0$. 
\item [{(2)}] Let $S$ be a symplectic surface intersecting $C$ at $p$, such that there is a Darboux complex coordinate system 
$(w_1,w_2)$ centered at $p$, in which $C$ and $S$ are given by $w_2=0$ and $w_2^n=aw_1^m$ for some relative prime 
integers $m,n>0$ and a complex number $a\neq 0$. Then the descendant $S^\prime$
of $S$ in the blow-down $(M^\prime,\omega^\prime)$ can be defined as follows: let $(z_1,z_2)$ be the complex coordinates 
of the standard symplectic $4$-ball in $(M^\prime,\omega^\prime)$, such that the complex line corresponding to the intersection 
point $p\in C$ is given by $z_1=0$, then $S^\prime$ is obtained by gluing a holomorphic disc to 
$S\setminus C$ defined by the equation $z_1^m=bz_2^{m+n}$ for some complex number 
$b\neq 0$, which is explicitly determined by $a$, $m$ and $n$.
\end{itemize}
\end{lemma}

\begin{proof}
Let the symplectic area of $C$ be $\omega(C)=\pi\delta_0^2$ for some $\delta_0>0$. Then by
the Weinstein neighborhood theorem, a neighborhood of $C$ in $(M,\omega)$ has a standard
model which we describe below. 

Let $(z_1,z_2)$ be the coordinates of $\C^2$ such that the standard symplectic structure $\omega_0$ is given by $\omega_0=\frac{i}{2}(dz_1\wedge d\bar{z}_1+dz_2\wedge d\bar{z}_2)$.
Let $B^4(\delta)=\{(z_1,z_2)||z_1|^2+|z_2|^2<\delta^2\}$ denote the open ball of radius $\delta>0$
in $\C^2$, and for any $\delta_1>\delta_0$, let $W(\delta_1)$ be the symplectic $4$-manifold which is obtained by collapsing the fibers of the Hopf fibration on the boundary of $B^4(\delta_1)\setminus B^4(\delta_0)$. Then a neighborhood of $C$ in $(M,\omega)$ is symplectomorphic to $W(\delta_1)$
for some $\delta_1$ where $\delta_1-\delta_0$ is sufficiently small. With this understood, the 
symplectic blow-down $(M^\prime,\omega^\prime)$ is obtained by cutting $(M,\omega)$ open
along $C$ and gluing in the standard symplectic $4$-ball $B^4(\delta_0)$ after fixing an
identification of a neighborhood of $C$ with $W(\delta_1)$. In the present situation, in order to
extend the symplectic surfaces $S_i\setminus C$ or $S\setminus C$ across the $4$-ball $B^4(\delta_0)$, we need to choose the identification of a neighborhood of $C$ with $W(\delta_1)$
more carefully. 

To this end, we consider the following reparametrization of a neighborhood of the circle 
$\{z_1=0\}\cap \s^3(\delta_0)$ in $\C^2$, where $\s^3(\delta_0)$ is the sphere of radius $\delta_0$,
by the map 
$$
(z_1,z_2)=(\frac{r\delta}{\sqrt{1+r^2}} e^{i(\theta+\phi)}, \frac{\delta}{\sqrt{1+r^2}} 
e^{i\phi}),
$$
for $0\leq r<r_0$, $\theta,\phi\in \R/2\pi \Z$, and $\delta$ lying in a small interval containing 
$\delta_0$. We note that $(r,\theta,\phi)$ gives a trivialization of the Hopf fibration near 
$z_1=0$ in $\s^3(\delta_0)$, with $(r,\theta)$ for the base and $\phi$ for the fiber. In the new
coordinates $(r,\theta,\delta,\phi)$, the standard symplectic structure on $\C^2$ takes the form
$$
\omega_0=\frac{r^2\delta}{1+r^2} d\delta\wedge d\theta+\frac{\delta^2r}{(1+r^2)^2}dr\wedge 
d\theta+\delta d\delta \wedge d\phi.
$$
Replacing $\delta^2$ by $\delta^2+\delta_0^2$ and assuming 
$0\leq \delta<\sqrt{\delta_1^2-\delta_0^2}$,
we obtain a description of the symplectic structure on $W(\delta_1)$ in a 
neighborhood of the image of $\{z_1=0\}\cap \s^3(\delta_0)$ in  $W(\delta_1)$
(where the image of $\{z_1=0\}\cap \s^3(\delta_0)$ has coordinates $\lambda=\delta=0$):
$$
\omega_0=\lambda d\lambda \wedge d\theta+\delta d\delta \wedge d\phi,
\mbox{ where } \lambda=\frac{r\sqrt{\delta^2+\delta_0^2}}{\sqrt{1+r^2}}.
$$
With this understood, the map $(w_1,w_2)=(\lambda e^{i\theta},\delta e^{i\phi})$ is a
symplectomorphism which identifies a neighborhood of the image of $\{z_1=0\}\cap \s^3(\delta_0)$ in  $W(\delta_1)$ 
with a neighborhood of $p\in C$ in $(M,\omega)$. Then by the relative version
of the Weinstein neighborhood theorem, we may extend this symplectomorphism to a 
symplectomorphism which identifies $W(\delta_1)$ with a neighborhood of $C$ in $(M,\omega)$.

With the preceding understood, we now consider case (1) of the lemma. First, note that the symplectic surface $S_0$ is given by
$w_1=0$ near the point $p$. Hence under the symplectomorphism $(w_1,w_2)=(\lambda e^{i\theta},\delta e^{i\phi})$ where 
$\lambda=\frac{r\sqrt{\delta^2+\delta_0^2}}{\sqrt{1+r^2}}$, the part of $S_0$ near $p$ as a symplectic surface in $W(\delta_1)$ 
is given by the equation $r=0$ in the coordinate system $(r,\theta,\delta,\phi)$, which implies that, as a symplectic surface in 
$\C^2$, it is given by the equation $z_1=0$. It follows immediately that one can extend $S_0\setminus C$ across the standard 
symplectic $4$-ball in $(M^\prime,\omega^\prime)$ by gluing in a complex disc contained in the complex line $z_1=0$. 
This is the descendant $S_0^\prime$ of $S_0$ in $(M^\prime,\omega^\prime)$. 

For each $i>0$, $S_i$ is given by the complex line $w_2=a_i w_1$ near the point $p$. Writing $a_i=\rho_i e^{i\kappa_i}$, we
parametrize $S_i$ near $p$ by the equations $w_1=te^{is}$ and $w_2=t\rho_i e^{i(s+\kappa_i)}$. Under the 
symplectomorphism $(w_1,w_2)=(\lambda e^{i\theta},\delta e^{i\phi})$ where $\lambda=\frac{r\sqrt{\delta^2+\delta_0^2}}{\sqrt{1+r^2}}$,
it is parametrized in the $(r,\theta,\delta,\phi)$ coordinate system by the following equations:
$$
r=\frac{t}{\sqrt{\delta_0^2+(\rho_i^2-1)t^2}},\; \theta=s,\; \delta=t\rho_i,\; \phi=s+\kappa_i.
$$
Now reviewing the part of $S_i$ near $p$ as a subset in $\C^2$, it is parametrized in the coordinates $(z_1,z_2)$ by the following
equations (recall we have replaced $\delta^2$ by $\delta_0^2+\delta^2$):
$$
z_1=\frac{r\sqrt{\delta_0^2+\delta^2}}{\sqrt{1+r^2}}e^{i(\theta+\phi)}=te^{i(2s+\kappa_i)}, \;\;
z_2=\frac{\sqrt{\delta_0^2+\delta^2}}{\sqrt{1+r^2}}e^{i\phi}=\sqrt{\delta_0^2+(\rho_i^2-1)t^2} 
\cdot e^{i(s+\kappa_i)}. 
$$
With this understood, we observe that $z_1,z_2$ satisfy the equation $z_1=b_i z_2^2$, where
$$
b_i=\frac{te^{-i\kappa_i}}{\delta_0^2+(\rho_i^2-1)t^2},
$$
for any $t>0$ which is sufficiently small. It is clear that $b_i\neq 0$ for each $i>0$, and that $\{a_i\}$ being distinct implies 
that $\{b_i\}$ are also distinct (for each fixed $t$). Now we fix a value $t_0>0$ which is sufficiently small, and 
remove the part $\{t\leq t_0\}$ from
$S_i$ and glue onto it the holomorphic disc defined by the equation $z_1=b_i z_2^2$, where
$$
b_i=\frac{t_0e^{-i\kappa_i}}{\delta_0^2+(\rho_i^2-1)t_0^2}.
$$
For $t_0$ small, one can smooth off the corners near the gluing region to obtain a symplectic surface in $(M^\prime,\omega^\prime)$,
which is defined to be the descendant $S_i^\prime$ of $S_i$ in the symplectic blow-down. This finishes the proof for case (1).

The argument for case (2) is similar. The surface $S$ near $p$ is given by the equation $w_2^n=aw_1^m$. Writing 
$a=\rho e^{i\kappa}$, we parametrize $S$ near $p$ by the equations 
$$
w_1=t^ne^{ins} \mbox{ and } w_2=t^m\rho^{\frac{1}{n}} e^{i(ms+\frac{\kappa}{n})}.
$$
Under the symplectomorphism $(w_1,w_2)=(\lambda e^{i\theta},\delta e^{i\phi})$ where 
$\lambda=\frac{r\sqrt{\delta^2+\delta_0^2}}{\sqrt{1+r^2}}$, it is parametrized in the $(r,\theta,\delta,\phi)$ coordinate system 
by the following equations:
$$
r=\frac{t^n}{\sqrt{\delta_0^2+\rho^{2/n} t^{2m}-t^{2n}}},\; \theta=ns,\; \delta=\rho^{1/n}t^m,\; \phi=ms+\kappa/n.
$$
In the coordinates $(z_1,z_2)$ on $\C^2$, the part of $S$ near $p$ is parametrized by the following equations: 
$$
z_1=t^n e^{i\kappa/n}\cdot e^{i(m+n)s}, \;\; z_2=\sqrt{\delta_0^2+\rho^{2/n} t^{2m}-t^{2n}} \cdot e^{i\kappa/n} e^{ims}.
$$
It follows easily that $z_1,z_2$ satisfy the equation $z_1^m=b z_2^{m+n}$, where 
$$
b=\frac{t^{mn} e^{-i\kappa}}{(\delta_0^2+\rho^{2/n} t^{2m}-t^{2n})^{\frac{m+n}{2}}} 
$$
for any $t>0$ which is sufficiently small. Clearly, $b\neq 0$. As in case (1), we fix a value $t_0>0$ sufficiently small, remove
the part $\{t\leq t_0\}$ from the surface $S$ and glue onto it the holomorphic disc (singular in this case) defined by the equation
$z_1^m=bz_2^{m+n}$, where in $b$ the variable $t$ is evaluated at $t_0$. The resulting surface (after smoothing off the corners)
is the descendant $S^\prime$ of $S$ in the symplectic blow-down $(M^\prime,\omega^\prime)$. This finishes the proof for case (2),
and the proof of the lemma is complete. 

\end{proof}

With Lemma 4.4 at hand, we shall define the descendant $D_{n-1}$ of $D_n$ in the next stage 
$(X_{n-1},\omega_{n-1})$ as follows.
First, since $n<l_s$ for each $s$, the classes $E_{l_s}$ all have been blown down in the earlier stages. We assume that for each $s$,
the class $E_{l_s}$ does not appear in any of the components of $D_N$ as the leading class
(i.e., this is the first time we cannot perturb the $(-1)$-sphere to a general position). 
With this understood, for each $s$, there is a point $\hat{E}_{l_s}$ and a small, standard symplectic $4$-ball $B(\hat{E}_{l_s})\subset X_n$ centered at $\hat{E}_{l_s}$, such that $\hat{E}_{l_s}\in C_n$ for each $s$, and the intersection $B(\hat{E}_{l_s})\cap C_n$ is a disc lying in a complex line (called a complex linear disc). 

\vspace{2mm}

{\bf Case (1):} Suppose the class $E_n$ does not appear in any of the components of $D_N$ which has zero $a$-coefficient. In this case, we can simply blow down $(X_n,\omega_n)$ 
along $C_n$ to the next stage $(X_{n-1},\omega_{n-1})$, which means that we will cut $X_n$ open along $C_n$ and then insert a standard symplectic $4$-ball of appropriate size. For any component $F_k$ in $D_n$ which intersects with $C_n$, there are two possibilities.
If an intersection point of $F_k$ with $C_n$ is inherited from the original intersection in $D_N$, then by the condition (\dag), there is no other component $F_l$ passing through this intersection point. For such an intersection point on $C_n$, we shall simply glue a disc to $F_k\setminus C_n$
which is lying on a complex line in the standard symplectic $4$-ball. Any other intersection point of $F_k$ with $C_n$ should occur at one of the points
$\hat{E}_{l_s}$. For any such intersection points, we shall define the descendant of $F_k$ in $X_{n-1}$ by extending the surface 
$F_k\setminus C_n$ across the standard symplectic $4$-ball according to Lemma 4.4(1). With this understood, we denote the center of the
standard symplectic $4$-ball by $\hat{E}_n$. Then it is easy to see that there is a small $4$-ball $B(\hat{E}_n)$ centered at $\hat{E}_n$,
such that each original intersection point on $C_n$ from $D_N$ determines a linear complex disc in $B(\hat{E}_n)$ as part of the descendant 
$D_{n-1}$, and each point $\hat{E}_{l_s}\in C_n$ determines a complex line in $B(\hat{E}_n)$ with the property that each linear complex disc in $B(\hat{E}_{l_s})\cap D_n$ which is not part of $C_n$ determines a holomorphic disc in $B(\hat{E}_n)$ as part of the descendant $D_{n-1}$, which has tangency of order $2$ with the complex line determined by the point $\hat{E}_{l_s}$. Finally, we remark that after shrinking the size, the $4$-ball $B(\hat{E}_n)$, particularly
the point $\hat{E}_n$, will survive to the last stage of the successive blowing-down. 
Note that $E_n\in\E_0(D_N)$, but $\forall s$, $E_{l_s}$ does not belong to $\E_0(D_N)$.

\vspace{2mm}

{\bf Case (2):} If the class $E_n$ appears in the expression of a symplectic sphere in $D_N$ whose $a$-coefficient is zero (note that in this case, $E_n$ is not an element of $\E_0(D_N)$), then more care is needed in defining the descendant $D_{n-1}$. And here is the reason: suppose $E_n$ is contained in $S_1$ whose $a$-coefficient is zero, and let $E_m$ be the leading class in $S_1$. Then $m<n$, and in a later stage of $(X_m,\omega_m)$ when we blow down the class $E_m$, we will be again in a situation where we cannot perturb the $(-1)$-sphere $C_m$ to a general position (because $C_m$ is the descendant of $S_1$ in $D_m$, so is part of $D_m$). In particular, we will have to apply Lemma 4.4 when blowing down the class $E_m$. With this understood, observe that in Lemma 4.4, near the point $p\in C$, the symplectic surfaces under consideration have to be in certain standard forms with respect to a complex coordinate system $(w_1,w_2)$ with standard symplectic structure, and in particular, the $(-1)$-sphere $C$ has to be given by a complex coordinate line $w_2=0$. 
This requires that, when we blow down the $(-1)$-sphere $C_n$, we need to arrange so that in the small $4$-ball $B(\hat{E}_n)\subset X_{n-1}$, the holomorphic discs $B(\hat{E}_n)\cap D_{n-1}$
can be placed in the model required in Lemma 4.4.

\vspace{2mm}

With this understood, we first make the following observation. 

\begin{lemma}
There are at most two components $F_k$ in $D_N$ such that 
$\left(1\right)$ the $a$-coefficient of $F_k$ is zero, 
$\left(2\right)$ the homological expression of $F_k$ contains the class $E_n$. Moreover, such a component $F_k$ can contain at most one of the classes $E_{l_s}$ in its homological expression, and the 
classes $E_{l_s}$ contained in two distinct such components $F_k$ must be distinct. 
{\em(}Recall $S=E_n-E_{l_1}-\cdots-E_{l_\alpha}$ is the symplectic sphere in $D_N$ that is under consideration.{\em)}
\end{lemma}

\begin{proof}
Suppose $S_1$ is such a component in $D_N$, i.e, the $a$-coefficient of $S_1$ is zero and the homological expression of $S_1$ contains the class $E_n$. Let $E_{j_1}$ be the leading class in $S_1$. Then the fact that $E_n$ is contained in $S_1$ implies that $j_1<n$ must be true. 
On the other hand, $S\cdot S_1\geq 0$ implies that 
$S\cdot S_1$, in fact, equals either $0$ or $1$. In the former case, $S_1$ contains exactly one of the classes $E_{l_s}$, and in the latter case, $S_1$ contains none of the classes $E_{l_s}$. 

Suppose $S_2$ is another such component in $D_N$, with $E_{j_2}$ being the leading class in $S_2$. Without loss of generality, we assume $j_2<j_1$. Then since $S_1,S_2$ both contain the class $E_n$, it follows easily from $S_1\cdot S_2\geq 0$ that $E_{j_1}$ must appear in the expression of $S_2$, the intersection $S_1\cdot S_2=0$, and the classes $E_{l_s}$ which are contained in $S_1,S_2$ must be distinct. With this understood, suppose to the contrary that there are more than two such components, and let $S_3$ be a third such component. Then the same argument as in the case of $S_2$ implies that the expression of $S_3$ must contain both $E_{j_1}$ and $E_n$. But then this would imply $S_2\cdot S_3<0$, which is a contradiction. The lemma follows easily from these considerations. 

\end{proof}

We shall consider separately according to the number of the symplectic spheres 
described in Lemma 4.5. 

\vspace{2mm}

{\bf Case (a):} Suppose there are two symplectic spheres $S_1,S_2\subset D_N$ with zero 
$a$-coefficient whose homological expressions contain the class $E_n$ in $S$. We shall need to make some very specific identification of a neighborhood of $C_n$ in $(X_n,\omega_n)$ with the standard model, which is described below. Assume $\omega_n(C_n)=\pi\delta_0$. 

Fix a coordinate system $(z_1,z_2)$ of $\C^2$ such that the standard symplectic structure 
$\omega_0$ on $\C^2$ is given by $\omega_0=\frac{i}{2}(dz_1\wedge d\bar{z}_1+dz_2\wedge d\bar{z}_2)$. Let $B^4(\delta)=\{(z_1,z_2)||z_1|^2+|z_2|^2<\delta^2\}$ denote the open ball of 
radius $\delta>0$ in $\C^2$, and for any $\delta_1>\delta_0$, let $W(\delta_1)$ be the symplectic 
$4$-manifold which is obtained by collapsing the fibers of the Hopf fibration on the boundary of $B^4(\delta_1)\setminus B^4(\delta_0)$. Then by the Weinstein neighborhood theorem, a neighborhood of $C_n$ in $(X_n,\omega_n)$ is symplectomorphic to $W(\delta_1)$
for some $\delta_1$ where $\delta_1-\delta_0$ is sufficiently small. With this understood, the 
symplectic blow-down $(X_{n-1},\omega_{n-1})$ is obtained by cutting $(X_n,\omega_n)$ open along $C_n$ and gluing in the standard symplectic $4$-ball $B^4(\delta_0)$ after fixing an
identification of a neighborhood of $C_n$ with $W(\delta_1)$. 

With the preceding understood, let $p_1,p_2$ be the intersection points of the descendants of $S_1,S_2$ in $D_n$ with $C_n$. Then by a relative version of the Weinstein neighborhood theorem,
we can choose an identification of a neighborhood of $C_n$ with $W(\delta_1)$ such that $p_1$ and 
$p_2$ are identified with the images of the Hopf fibers at $z_1=0$ and $z_2=0$ respectively. 
With this understood, when we apply Lemma 4.4 to the points $p_1,p_2$, we can furthermore arrange
the descendants of $S_1,S_2$ in $D_n$ to be the symplectic surface $S_0$ in Lemma 4.4, so that after applying Lemma 4.4, the descendants of $S_1,S_2$ in $D_{n-1}\cap B^4(\delta_0)$ are given by
the complex lines $z_1=0$ and $z_2=0$ respectively. Moreover, any other component of $D_n$ which intersects $C_n$ at either $p_1$ or $p_2$ will have its descendant in $D_{n-1}$ given by a holomorphic disc in $B^4(\delta_0)$ of the form $z_1=bz_2^2$ or $z_2=bz_1^2$ respectively (more generally, of the form $z_1^m=b z_2^{m+n}$ if before blowing down it is given by $w_2^n=aw_1^m$,
etc. ). It remains to deal with the intersection points $\hat{E}_{l_s}\in C_n$ which are not $p_1,p_2$.
By the assumption (a) in Theorem 4.3, for any such an $\hat{E}_{l_s}$, there is only one component in $D_n$ which intersects $C_n$ at $\hat{E}_{l_s}$, with intersection number $+1$. (Equivalently, there is only one holomorphic disc in the small $4$-ball $B(\hat{E}_{l_s})$ which does not lie in $C_n$.) 
By a small perturbation, we can arrange this component to coincide with the fiber at 
$\hat{E}_{l_s}\in C_n$ in $W(\delta_1)$,  so that it can be extended across the $4$-ball $B^4(\delta_0)$ by a linear complex disc (given by equation $z_2=az_1$) when we blow down 
$C_n$. In summary, the holomorphic discs $B^4(\delta_0)\cap D_{n-1}$ can be placed in a model that is required in Lemma 4.4 before the blowing down, so that in a later stage, 
when we blow down the $(-1)$-sphere which is the descendant of $S_1$ or $S_2$,
Lemma 4.4 can be applied in the process.

\vspace{2mm}

{\bf Case (b):} Suppose there is only one symplectic sphere $S_1\subset D_N$ with zero 
$a$-coefficient whose homological expression contains the class $E_n$ in $S$. Let $p_1$ be
the intersection point of $C_n$ with the descendant of $S_1$ in $D_n$. Then by the assumption (b) in Theorem 4.3, there is at most one intersection point $\hat{E}_{l_s}\neq p_1$ such that the small
$4$-ball $B(\hat{E}_{l_s})$ contains more than one holomorphic discs which do not lie in $C_n$.
With this understood, we shall choose an identification of a neighborhood of $C_n$ in 
$(X_n,\omega_n)$ with $W(\delta_1)$ such that $p_1$ and 
the intersection point $\hat{E}_{l_s}$ are identified with the images of the Hopf fibers at $z_1=0$ and $z_2=0$ respectively. Then by the same argument as in Case (a), we can arrange such that 
the holomorphic discs $B^4(\delta_0)\cap D_{n-1}$ can be placed in an appropriate model, so that when we blow down the $(-1)$-sphere which is the descendant of $S_1$ in
a later stage, Lemma 4.4 can be applied in the process.

\vspace{2mm}

With the preceding understood, it follows easily that under assumptions (a) and (b), one can continue the process and successively blow down the classes $E_N, E_{N-1}, \cdots, E_2$ to reach to the stage $(X_1,\omega_1)$ (where $X_1=\C\P^2\# \overline{\C\P^2}$), obtaining a canonically 
constructed descendant $D_1$ of $D_N$ in $(X_1,\omega_1)$. We remark that there is an
$\omega_1$-compatible almost complex structure $J_1$, such that $D_1$ is $J_1$-holomorphic. 

It remains to show that if any of the conditions (c), (d), (e) is satisfied, then one can further blow
down the class $E_1$ to reach $\C\P^2$ in the final stage. First, assume (c) is true. In this case,
since $\omega_N(E_1)=\omega_N(E_2)$, the class $E_1$ also has the minimal area in 
$(X_2,\omega_2)$, so that we can represent both $E_1,E_2$ by a $J_2$-holomorphic sphere.
It follows that we can blow down both $(-1)$-classes at the same time.

Next, suppose condition (d) is satisfied. In this case, there is a symplectic sphere $S$ in $D_N$ such
that $E_1$ appears in the expression of $S$ as the leading class. We observe that the descendant of $S$ in $D_1$ is a symplectic $(-1)$-sphere representing the class $E_1$. We simply blow down $(X_1,\omega_1)$ along this $(-1)$-sphere to reach the final stage $\C\P^2$.

Finally, suppose condition (e) is satisfied. In this case, we appeal to Lemma 2.3 of \cite{C2},
which says that either $E_1$ is represented by a $J_1$-holomorphic sphere, or there is a 
$J_1$-holomorphic sphere $C$ such that $E_1=m(H-E_1)+C$ for some $m\geq 1$. 
In the former case, we can blow down the class $E_1$. In the latter case, we reach a contradiction as follows. By condition (e), there is a component $F_k$ of $D_N$ whose $a$-coefficient, $a$, and the
$b_i$-coefficient for $E_1$, $b$, obeys $2b<a$. Let $\hat{F}_k$ denote the descendant of $F_k$
in $D_1$, which is $J_1$-holomorphic and has class $aH-bE_1$. Then we have
$$
0\leq C\cdot \hat{F}_k=(m+1)b-ma,
$$
contradicting the assumption $2b<a$ and the fact $m\geq 1$. The proof of Theorem 4.3 is complete.

\subsection{Examples} 
For the purpose of illustration, we shall apply the successive symplectic blowing down procedure to 
some concrete examples, where $X_N$ is the resolution $\tilde{X}$ of the symplectic 
$4$-orbifold $X$ and $D_N=D$, the pre-image of the singular set of $X$ under the map 
$\tilde{X}\rightarrow X$.

\begin{example}
(1) Consider the case where $X$ has a singular set described in (i) of Theorem 1.2(2), i.e., the singular set consists of $9$ isolated non-Du Val singularities of isotropy of order $3$. In this case, the symplectic configuration $D$ is a disjoint union of $9$ symplectic $(-3)$-spheres, to be denoted by $F_1,F_2, \cdots,F_9$. Note that the canonical class of the resolution $\tilde{X}$ is given by 
$$
c_1(K_{\tilde{X}})=-\frac{1}{3}(F_1+F_2+\cdots+F_9).
$$
It follows immediately that $\tilde{X}=\C\P^2\# 12\overline{\C\P^2}$.

The following is a set of possible homological expressions for $F_1,F_2,\cdots,F_9$:
\begin{itemize}
\item $H-E_{i}-E_{r}-E_{s}-E_t$, $H-E_{i}-E_u-E_v-E_w$, $H-E_{i}-E_x-E_y-E_z$,
\item $H-E_{j}-E_{r}-E_{u}-E_x$, $H-E_{j}-E_s-E_v-E_y$, $H-E_{j}-E_t-E_w-E_z$,
\item $H-E_{k}-E_{r}-E_{v}-E_z$, $H-E_{k}-E_s-E_w-E_x$, $H-E_{k}-E_t-E_u-E_y$.
\end{itemize}
Each class can be represented by a symplectic $(-3)$-sphere, each pair of distinct classes 
has zero intersection number, and the sum of the $9$ classes equals $-3c_1(K_{\tilde{X}})$.
Furthermore, one can arrange so that the symplectic structure on $\tilde{X}$ is odd, 
e.g., when $F_1,F_2, \cdots,F_9$ have the same area (cf. \cite{C}, Lemma 4.1). 

It is easy to see that the assumptions (a) and (b) are satisfied, and also, the condition (e) is satisfied. 
Furthermore, the set $\E_0(D)$ consists of all the $12$ $E_i$-classes. 
Thus by the successive blowing-down procedure, we obtain a symplectic arrangement $\hat{D}$ in
$\C\P^2$, which is a union of $9$ symplectic lines (i.e., a symplectic sphere of degree $1$) intersecting at $12$ points. Note that each line contains $4$ intersection points, each intersection point is contained in $3$ lines, so $\hat{D}$ has an incidence relation which is the same as that of the dual configuration of the famous Hesse configuration (cf. \cite{H}). In particular, $\hat{D}$ can be realized by an arrangement of complex lines.

(2) Consider the case where $X$ has a singular set as in (ii) of Theorem 1.2(2). In this case,
$D$ is a disjoint union of $5$ pairs of a symplectic $(-3)$-sphere and a symplectic $(-2)$-sphere,
denoted by $F_{1,k}, F_{2,k}$ for $k=1,2,\cdots,5$, where each pair of symplectic spheres 
$F_{1,k}$, $F_{2,k}$ intersect transversely and positively in one point. Moreover,
$$
c_1(K_{\tilde{X}})=-\frac{1}{5}\sum_{k=1}^5 (2F_{1,k}+F_{2,k}). 
$$
It follows easily that $\tilde{X}=\C\P^2\# 11\overline{\C\P^2}$.

The following is a set of possible homological expressions for $F_{1,k},F_{2,k}$, $1\leq k\leq 5$:
\begin{itemize}
\item $F_{1,1}=H-E_{i_1}-E_{i_2}-E_{i_3}-E_{i_4}$, $F_{2,1}=H-E_r-E_{i_5}-E_{i_{10}}$,
\item $F_{1,2}=H-E_{i_1}-E_{i_5}-E_{i_6}-E_{i_7}$, $F_{2,2}=H-E_r-E_{i_3}-E_{i_{9}}$,
\item $F_{1,3}=H-E_{i_2}-E_{i_5}-E_{i_8}-E_{i_9}$, $F_{2,3}=H-E_r-E_{i_4}-E_{i_{6}}$,
\item $F_{1,4}=H-E_{i_3}-E_{i_6}-E_{i_8}-E_{i_{10}}$, $F_{2,4}=H-E_r-E_{i_2}-E_{i_7}$,
\item $F_{1,5}=H-E_{i_4}-E_{i_7}-E_{i_9}-E_{i_{10}}$, $F_{2,5}=H-E_r-E_{i_1}-E_{i_{8}}$,
\end{itemize}
where the symplectic structure on $\tilde{X}$ can be arranged so that it is odd, e.g., by requiring 
that the symplectic spheres $F_{1,k}, F_{2,k}$, where $k=1,2,\cdots,5$, have the same area 
(cf. \cite{C}, Lemma 4.1). 
Again, the assumptions (a), (b) and the condition (e) are satisfied, so we can blow down $\tilde{X}$ and transform $D$ to a symplectic arrangement $\hat{D}\subset \C\P^2$. In this case, $\hat{D}$ is also a symplectic line arrangement, consisting of $10$ lines which intersect at $16$ points. There are $5$ original intersection points, i.e., those inherited from $D$, and $11$ new intersection points corresponding to the $11$ $E_i$-classes. The original intersection points are double points, and among the 
$11$ new intersection points, $10$ are triple points and one point is contained in $5$ lines. 
We note that this incidence relation is realized by the real line arrangement $A_1(2m)$ for $m=5$.
(Recall that $A_1(2m)$, for $m\geq 3$, is the arrangement of $2m$ lines in $\R\P^2$, of which $m$ are the lines determined by the edges of a regular $m$-gon in $\R^2$, while the other $m$ are the lines of symmetry of that $m$-gon, cf. \cite{H}.) In particular, $\hat{D}$ can be realized by the complexification of a real line arrangement.
\end{example}

\begin{example}
Here we consider the orbifold $X$ in Example 4.6(1) again, but with the following possible set of
homological expressions for $F_1,F_2,\cdots,F_9$:
\begin{itemize}
\item $F_1=E_u-E_j-E_w$, $F_2=E_y-E_k-E_z$,
\item $F_3=H-E_{i}-E_{r}-E_{s}-E_t$, $F_4=H-E_{i}-E_u-E_v-E_w$,
\item $F_5= H-E_{i}-E_x-E_y-E_z$, $F_6=H-E_{j}-E_{r}-E_{u}-E_x$, 
\item $F_7=H-E_{k}-E_r-E_v-E_y$, $F_8=2H-E_s-E_t-E_u-E_y-E_j-E_v-E_z$, 
\item $F_9=2H-E_s-E_t-E_u-E_y-E_k-E_x-E_w$.
\end{itemize}
Again, the assumptions (a), (b) and the condition (e) are satisfied. In this case, the symplectic
arrangement $\hat{D}$ is a union of $5$ symplectic lines and $2$ symplectic spheres of degree $2$,
consisting of the descendants of $F_k$ for $3\leq k\leq 9$. As for the intersection points, note that 
$\E_0(D)=\{E_s, E_t, E_x, E_v, E_r, E_i, E_u, E_y\}$, so there are totally $8$ intersection points labelled by these classes. Moreover, each of the $6$ intersection points $\hat{E}_s$, $\hat{E}_t$, 
$\hat{E}_x$, 
$\hat{E}_v$, $\hat{E}_r$, $\hat{E}_i$ is a triple point; it is contained in $3$ components in $\hat{D}$
intersecting at it transversely. As for $\hat{E}_u$ and $\hat{E}_y$, let's denote by $\hat{F}_k$ the descendant of $F_k$ in $\hat{D}$, for $3\leq k\leq 9$. Then the class $E_j$ (resp. $E_w$) determines a complex line in the $4$-ball $B(\hat{E}_u)$, such that $\hat{F}_6$ and $\hat{F}_8$ 
(resp. $\hat{F}_4$ and $\hat{F}_9$) are tangent to it at $\hat{E}_u$,
with the intersection of $\hat{F}_6$ and $\hat{F}_8$ (resp. $\hat{F}_4$ and $\hat{F}_9$) 
at $\hat{E}_u$ being of tangency of order $2$ (cf. Remark (2)(ii) following Theorem 4.3). Similar discussions apply to $\hat{E}_y$, and the
classes $E_k, E_z$. 

\end{example}


\vspace{2mm}

{\Small University of Massachusetts, Amherst.\\
{\it E-mail:} wchen@math.umass.edu

\end{document}